\documentclass[11pt]{amsart}

\usepackage{amsmath}
\usepackage{amssymb}
\usepackage{amsfonts}

\usepackage{amsthm}
\usepackage{enumerate}

\usepackage[mathscr]{eucal}
\usepackage{mathrsfs}

\usepackage{bbm}

\usepackage{subfigure}
\usepackage{tikz}
\usepackage{graphicx}
\usepackage{verbatim}

\newcommand{\vsig}{\varsigma}

\newcommand{\Be}{\begin{equation}}
\newcommand{\Ee}{\end{equation}}

\newcommand{\Ba}[1]{\begin{array}{#1}}
\newcommand{\Ea}{\end{array}}
\newcommand{\Bea}{\begin{eqnarray}}
\newcommand{\Eea}{\end{eqnarray}}
\newcommand{\Beas}{\begin{eqnarray*}}
\newcommand{\Eeas}{\end{eqnarray*}}
\newcommand{\Benu}{\begin{enumerate}}
\newcommand{\Eenu}{\end{enumerate}}
\newcommand{\Bi}{\begin{itemize}}
\newcommand{\Ei}{\end{itemize}}

\def\intslash{\rlap{\kern  .32em $\mspace {.5mu}\backslash$ }\int}
\def\qsl{{\rlap{\kern  .32em $\mspace {.5mu}\backslash$ }\int_{Q_x}}}

\newcommand {\Span} {\operatorname{span}}

\def\emph#1{{\it #1 }}

\def\ga{\gamma}

\def\cf{{\it cf}}

\def\dist{{\text{\it dist}}}

\def\supp{{\text{\rm supp}}}

\def\inn#1#2{\langle#1,#2\rangle}

\def\noi{\noindent}

\def\card{\text{\rm card}}
\def\lc{\lesssim}

\def\eps{\varepsilon}
\def\ep{\epsilon}

\def\ka{\kappa}
            
             \def\La{\Lambda}

\def\om{\omega}

\def\fA{{\mathfrak {A}}}

\def\fM{{\mathfrak {M}}}

\def\fP{{\mathfrak {P}}}

\def\fZ{{\mathfrak {Z}}}

\def\fa{{\mathfrak {a}}}

\def\bbE{{\mathbb {E}}}

\def\bbN{{\mathbb {N}}}

\def\bbR{{\mathbb {R}}}

\def\bbZ{{\mathbb {Z}}}

\def\sD{{\mathscr {D}}}

\def\sH{{\mathscr {H}}}

\def\cA{{\mathcal {A}}}

\def\cE{{\mathcal {E}}}
\def\cF{{\mathcal {F}}}

\def\cS{{\mathcal {S}}}

\def\cU{{\mathcal {U}}}

\def\cY{{\mathcal {Y}}}
\def\cZ{{\mathcal {Z}}}

\def\tI{{\widetilde{I}}}

 at 10 true pt

\def\be#1{\begin{equation}\label{ #1}}
\def\endeq{\end{equation}}
\def\endal{\end{align}}
\def\bas{\begin{align*}}
\def\eas{\end{align*}}
\def\bi{\begin{itemize}}
\def\ei{\end{itemize}}

\def\eps{\varepsilon}

\def\emph#1{{\it #1}}
\def\textbf#1{{\bf #1}}

\def\bbone{{\mathbbm 1}}

\theoremstyle{plain}
  \newtheorem{theorem}{Theorem}[section]

   \newtheorem{proposition}[theorem]{Proposition}
   
   \newtheorem{lemma}[theorem]{Lemma}
   \newtheorem{corollary}[theorem]{Corollary}

\theoremstyle{remark}
   \newtheorem{remark}[theorem]{Remark}

\theoremstyle{definition}
   \newtheorem{definition}[theorem]{Definition}

\newcommand {\SE} {{\mathbb E}}
\newcommand {\SN} {{\mathbb N}}
\newcommand {\SR} {{\mathbb R}}

\newcommand {\SZ} {{\mathbb Z}}
\newcommand {\SNz} {{\mathbb \SN_0}}
\newcommand {\SRd} {{\mathbb R^d}}
\newcommand {\SZd} {{\mathbb Z^d}}

\newcommand {\e} {{\varepsilon}}

\newcommand{\dt}{{\delta}}

\newcommand {\mand} {{\quad\mbox{and}\quad}}
\renewcommand {\mid} {{\,\,\,\colon\,\,\,}}
\def\supp{\mathop{\rm supp}}
\def\dist{\mathop{\rm dist}}

\def\sign{\mathop{\rm sign}}

\newcommand {\ProofEnd} {
             \begin{flushright} \vskip -0.2in $\Box$ \end{flushright}}

\newcommand {\Ds} {\displaystyle}

\newcounter{reg}
\newcommand{\sline}{{\smallskip

\noindent}}

\newcommand{\bline}{{\medskip

\noindent}}

\def\Xint#1{\mathchoice
{\XXint\displaystyle\textstyle{#1}}%
{\XXint\textstyle\scriptstyle{#1}}%
{\XXint\scriptstyle\scriptscriptstyle{#1}}%
{\XXint\scriptscriptstyle\scriptscriptstyle{#1}}%
\!\int}
\def\XXint#1#2#3{{\setbox0=\hbox{$#1{#2#3}{\int}$ }
\vcenter{\hbox{$#2#3$ }}\kern-.6\wd0}}

\def\mint{\Xint-}

\begin{document}

\title
[Haar system in Triebel-Lizorkin spaces: endpoint results]
{The Haar system in Triebel-Lizorkin spaces:  
 endpoint results}

\author[G. Garrig\'os \ \ \ A. Seeger \ \ \ T. Ullrich] {Gustavo Garrig\'os   \ \ \ \   Andreas Seeger \ \ \ \ Tino Ullrich}

\address{Gustavo Garrig\'os\\ Department of Mathematics\\University of Murcia\\30100 Espinardo\\Murcia, Spain} \email{gustavo.garrigos@um.es}

\address{Andreas Seeger \\ Department of Mathematics \\ University of Wisconsin \\480 Lincoln Drive\\ Madison, WI,53706, USA} \email{seeger@math.wisc.edu}
\address{Tino Ullrich\\ Fakult\"at f\"ur Mathematik\\ Technische Universit\"at  Chemnitz\\09107 Chemnitz, Germany}
\email{tino.ullrich@mathematik.tu-chemnitz.de}

\subjclass[2010]{46E35, 46B15, 42C40}
\keywords{Schauder basis, basic sequence, unconditional basis, dyadic averaging operators, Haar system, Sobolev and Besov spaces,  Triebel-Lizorkin spaces}

\maketitle

{\centering\footnotesize Dedicated to Guido Weiss, with affection, on his 91st birthday\par}


\begin{abstract}
We characterize the Schauder and  unconditional basis properties for the Haar system
in the Triebel-Lizorkin spaces $F^s_{p,q}(\SR^d)$, at the endpoint cases $s=1$, $s=d/p-d$ and $p=\infty$.
Together with the earlier results in \cite{su,gsu}, this completes the picture for such properties in the Triebel-Lizorkin scale,
and complements a similar study for the Besov spaces given in \cite{gsu-endpt}. 
\end{abstract}


\section{Introduction and statements of main results}

In this paper we essentially complete the study of the basis properties for the (inhomogeneous) Haar system in the scale of Triebel-Lizorkin spaces $F^s_{p,q}(\bbR^d)$. 
In particular, we describe the behavior at the endpoint cases which was left open in our  earlier work \cite{gsu}.
Similar endpoint questions for the family of Besov spaces have been presented in the companion paper \cite{gsu-endpt}. 
We note that markedly different outcomes occur for each family, in both the non-endpoint situations (\cite{triebel78, triebel-bases, su, sudet, gsu})
and the endpoint (\cite{gsu-endpt}, \cite{oswald}) situations.
 
We now set the basic notation required to state the results. Given the one variable functions
 $h^{(0)}=\bbone_{[0,1)}$ and $h^{(1)}=\bbone_{[0,1/2)}-\;\bbone_{[1/2,1)}$,
for each $\ep=(\ep_1,\ldots,\ep_d)\in\{0,1\}^d$, $k\in \bbN_0$ and $\nu=(\nu_1,\dots, \nu_d)\in \bbZ^d$, we define
\[
h^{\ep}_{k,\nu}(x)
:= \prod_{i=1}^d h^{(\ep_i)}(2^kx_i-\nu_i),\quad x=(x_1,\ldots,x_d)\in\SR^d.
\]
Then, the \emph{Haar system} is the collection of functions 
\[
\sH_d=\Big\{h^{\vec 0}_{0,\nu}\Big\}_{\nu \in\SZd}\cup \Big\{h^{\ep}_{k,\nu}\mid k\in\SNz,\;\nu\in\SZd,\;\ep\in\Upsilon\Big\},
\]
where we denote $\Upsilon=\{0,1\}^d\setminus\{\vec 0\}$. 

Consider $F^s_{p,q}(\SR^d)$ with the usual definition in \cite[\S2.3.1]{Tr83} or \cite[\S12]{FrJa90}. 
To investigate the Schauder basis properties of $\sH_d$, we initially assume that
 $0<p,q<\infty$ (so that $\cS$ is dense in $F^s_{p,q}$, and the latter is separable), 
and that
\Be
 h^\ep_{k,\nu}\in F^s_{p,q}\mand h^\ep_{k,\nu}\in (F^s_{p,q})^*,\quad\forall\,\ep,k,\nu.
\label{BBdual}
\Ee
Given an enumeration  $\cU=\big\{u_n=h^{\ep(n)}_{k(n),\nu(n)}\big\}_{n=1}^\infty$ of $\sH_d$, we consider the corresponding partial sum operators 
\Be
S_Rf=\label{SR} S_R^\cU f= \sum_{n=1}^R 
u_n^*(f) u_n \,,\quad R\in\SN,
\Ee
where  the linear functionals $u_n^*$ are defined by
\Be
u_n^* (f) =
2^{k(n)d} \langle f,h_{k(n),\nu(n)}^{\ep(n)}\rangle \,,\quad f\in\cS.
\label{unstar}
\Ee
The condition in \eqref{BBdual} ensures that these operators are well-defined and individually bounded
in $F^s_{p,q}(\SR^d)$. Also, $u^*_n(u_m)=\delta_{n,m}$, $n,m\geq1$.

The basis properties of $\cU$ are related to the validity of the bound \Be
\label{uniformSRbound}
  \sup_{R\in\SN}\|S_R^\cU\|_{F^s_{p,q}\to F^s_{p,q}}<\infty.
\Ee 
Indeed, if  $\text{span} \,\sH_d $ is dense in $F^s_{p,q}$, then 
\eqref{uniformSRbound} is equivalent to  $\cU$ being a \emph{Schauder basis} of $F^s_{p,q}$, that is 
\Be\label{basicseqdef}\lim_{R\to\infty}\|S_R^\cU f-f\|_{F^s_{p,q}}=0
\Ee
for every $f\in F^s_{p,q}$. 
Moreover, the basis is \emph{unconditional} if and only if the bound in \eqref{uniformSRbound} is uniform in all enumerations $\cU$.
Finally, if $ \Span\sH_d$ is not assumed to be dense, then \eqref{uniformSRbound} still implies that $\cU$ is a \emph{basic sequence} of $F^s_{p,q}$, meaning that
\eqref{basicseqdef} holds for all $f$ in the $F^s_{p,q}$-closure of $ \Span\sH_d$.

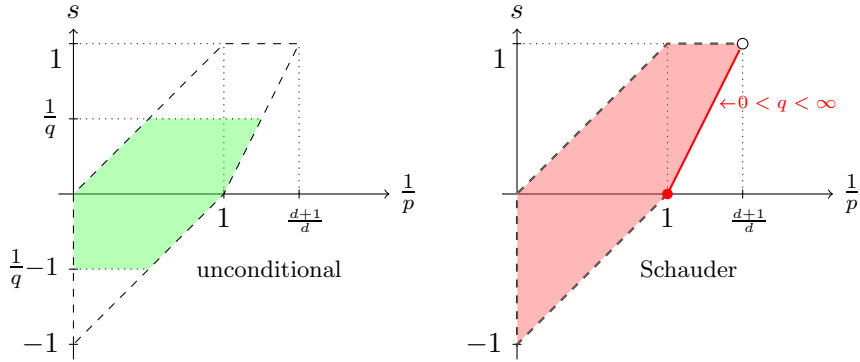
\begin{figure}[h]
 \centering
\subfigure
{\begin{tikzpicture}[scale=2]

\node [right] at (0.75,-0.5) {{\footnotesize unconditional}};

\draw[->] (-0.1,0.0) -- (2.1,0.0) node[right] {$\frac{1}{p}$};
\draw[->] (0.0,-0.0) -- (0.0,1.1) node[above] {$s$};
\draw (0.0,-1.1) -- (0.0,-1.0)  ;

\draw (1.0,0.03) -- (1.0,-0.03) node [below] {$1$};
\draw (1.5,0.03) -- (1.5,-0.03) node [below] {{\tiny $\;\;\frac{d+1}{d}$}};
\draw (0.03,1.0) -- (-0.03,1.00);
\node [left] at (0,0.9) {$1$};
\draw (0.03,.5) -- (-0.03,.5) node [left] {$\tfrac{1}{q}$};
\draw (0.03,-.5) -- (-0.03,-.5) node [left] {$\tfrac{1}{q}${\small{$-1$}}};
\draw (0.03,-1.0) -- (-0.03,-1.00) node [left] {$-1$};

\draw[dotted] (1.0,0.0) -- (1.0,1.0);
\draw[dotted] (0,1.0) -- (1.0,1.0);
\draw[dotted] (1.5,0.0) -- (1.5,1.0);

\path[fill=green!70, opacity=0.4] (0.0,0.0) -- (.5,.5)-- (1.25,0.5) -- (1,0)--(.5,-.5) -- (0,-0.5)--(0,0);
\draw[dotted] (0,0.5)--(1.25,0.5);
\draw[dotted] (0,-0.5)--(0.5,-0.5);

\draw[dashed] (0.0,-1.0) -- (0.0,0.0) -- (1.0,1.0) -- (1.5,1.0) -- (1.0,0.0) --
(0.0,-1.0);

\end{tikzpicture}
}
\subfigure
{
\begin{tikzpicture}[scale=2]

\node [right] at (0.75,-0.5) {{\footnotesize Schauder}};

\draw[->] (-0.1,0.0) -- (2.1,0.0) node[right] {$\frac{1}{p}$};
\draw[->] (0.0,-0.0) -- (0.0,1.1) node[above] {$s$};
\draw (0.0,-1.1) -- (0.0,-1.0)  ;

\draw (1.0,0.03) -- (1.0,-0.03) node [below] {$1$};
\draw (1.5,0.03) -- (1.5,-0.03) node [below] {{\tiny $\;\;\frac{d+1}{d}$}};
\draw (0.03,1.0) -- (-0.03,1.00);
\node [left] at (0,0.9) {$1$};
\draw (0.03,-1.0) -- (-0.03,-1.00) node [left] {$-1$};

\draw[dotted] (1.0,0.0) -- (1.0,1.0);
\draw[dotted] (0,1.0) -- (1.0,1.0);
\draw[dotted] (1.5,0.0) -- (1.5,1.0);

\draw[dashed, thick] (1,0) -- (0.0,-1.0)--(0.0,0.0) -- (1,1)--(1.5,1);


\draw[white, fill=red!70, opacity=0.4] (0,0) -- (1,1)
-- (1.5,1.0) -- (1.0,0) --(0,-1)--(0,0);

\draw[thick,red] (1.48,0.96) -- (1.0,0.0);
\fill[red] (1,0) circle (1pt);
\fill[white] (1.5,1) circle (1pt);
\draw (1.5,1) circle (1pt);
\draw[<-, thin, red] (1.35,0.6)--(1.45,0.6);
\node [red, right] at (1.4,0.6) {{\tiny $ {0<q<\infty}$}};


\end{tikzpicture}

}
\caption{Parameter domain $\fP$  for  $\sH_d$ in $F^s_{p,q}(\SRd)$. The left region corresponds to unconditionality, and right region to the Schauder basis property.}\label{fig1}
\end{figure}

The pentagon $\fP$ depicted in Figure \ref{fig1} shows the natural index region for these problems; outside its closure either \eqref{BBdual} or the density of $\Span\sH_d$ fail.
The open pentagon corresponds to the range
 $\frac{d}{d+1}<p<\infty$,  $0<q <\infty$,  and  \Be
\max\Big\{d(\tfrac 1p-1),\tfrac 1p-1\Big\} <s< \min\Big\{1,\tfrac 1p\Big\}.\label{range1}\Ee  
Triebel showed in \cite[Theorem 2.21]{triebel-bases} that $\sH_d$ is an unconditional basis of $F^s_{p,q}(\SR^d)$ in the green shaded region, where the additional restriction
\Be
 \max\big\{d(\tfrac1q-1),\tfrac1q-1\big\}<s<\tfrac1q\,
\label{q_range}
\Ee
is imposed. The necessity of condition \eqref{q_range} for unconditionality was established in \cite{su, sudet} (for $d=1$).
On the other hand, we recently showed in \cite{gsu} that natural enumerations of $\sH_d$ form a Schauder basis of $F^s_{p,q}(\SR^d)$ in the full open pentagon $\fP$. 
Except for a few trivial cases, the behavior at the points $(1/p,s)$ lying in the boundary of $\fP$ was left unexplored.

In this paper we 
attempt to 
fill this gap giving an answer with a  positive or negative outcome depending on the secondary index $q$. 
Moreover, when possible, the negative answer is replaced by a suitable basic sequence property.

We first state the complete range for unconditionality, which contains new negative cases and a multivariate extension of the examples in \cite{su}.

\begin{theorem} \label{th_unc}
Let $0<p,q<\infty$ and $s \in \bbR$. 
Then, $\sH_d $ is an unconditional basis of $F^s_{p,q}(\SR^d)$ if and only if
 the conditions \eqref{range1} and \eqref{q_range} are both satisfied.
\end{theorem}

In the next results we drop unconditionality, and consider the Schauder basis property
for the following natural orderings of the Haar system $\sH_d$; see \cite{gsu, gsu-endpt}.

 \smallskip
 
\begin{definition} \label{strongly-adm}
  {\em (i)} An enumeration $\cU$  is  said to be {\it admissible} if for some constant $b\in \bbN$
the following holds: for each cube $I_{\nu} = \nu+[0,1]^d$, $\nu\in \bbZ^d$, if $u_n$ and $u_{n'}$ are both supported in $I_\nu$ and $|\supp(u_n)|\geq 2^{bd}|\supp(u_{n'})|$, then  necessarily $n<n'$\,.

{\em (ii)} 
 { An enumeration $\cU$ is {\it strongly admissible} if for some constant $b\in \bbN$ the following holds:
 for each cube $I_\nu$, $\nu\in\SZ^d$, if $I^{**}_\nu$ denotes the five-fold dilated cube with respect to its center, and if $u_n$ and $u_{n'}$  are supported in  $I^{**}_\nu$ with $|\supp (u_n)| \ge 2^{bd}|\supp( u_{n'}) |$ then necessarily $n< n'$.}

\end{definition}

Our next theorem characterizes the Schauder basis property in $F^s_{p,q}$ for the class of strongly admissible enumerations of $\sH_d$.
A new positive result is obtained in the line $s=d/p-d$, when $\frac{d}{d+1}<p\leq 1$; see Figure \ref{fig1}. 
The special case $F^0_{1,2}=h^1$ is classical, and was established in \cite{billard, woj1}. The negative results for $s=1$ are also new. 

   \begin{theorem} \label{th1} Let $0<p,q\leq \infty$ and $s \in \bbR$. Then, the following statements are equivalent, i.e. (a)$\iff$(b):
  
 \sline  (a) 
Every  strongly admissible  enumeration $\cU$ of $\sH_d $ is a  Schauder basis of $F^{s}_{p,q}(\bbR^d)$.

\bline (b) One  of the following three  conditions is satisfied:

\Benu \item[(i)\;\;] $\quad 1< p< \infty$, $\quad\; \frac 1p -1< s< \frac1p$, $\quad 0<q< \infty$,

\item[(ii)\;]  $\quad \frac{d}{d+1} < p\leq   1$, $\quad \frac dp-d<s<1$, $\quad 0<q< \infty$,

\item[(iii)] $\quad \frac{d}{d+1} < p\le 1$, $\quad s=\frac dp-d$, $\quad 0<q< \infty$.
\Eenu
\end{theorem} 

\

As in \cite{gsu, gsu-endpt}, a crucial tool in our analysis will be played by the \emph{dyadic averaging operators} $\bbE_N$. 
That is, if $\sD_N$ is the set of all dyadic cubes of length $2^{-N}$,
 \[I_{N,\nu }=2^{-N}(\nu +[0,1)^d),\quad \nu \in\SZ^d,\] 
then we define 
\Be\label{expect}\bbE_N f(x)=\sum_{\nu \in \bbZ^d} \bbone_{I_{N,\nu }}(x)\,2^{Nd} \,\int_{I_{N,\nu }}f(y) dy\,,\Ee
at least for $f\in\cS$. We shall also need the following companion operators involving Haar functions of a fixed frequency level.
Namely, for  $N\in \bbN$ and any $\fa=(a_{\nu,\e})_{\nu,\e}\in \ell^\infty (\bbZ^d\times \Upsilon)$ we set
\Be\label{TNadef}
T_N[f,\fa] =\sum_{\ep\in \Upsilon} \sum_{\nu\in \bbZ^d} a_{\nu,\ep} 2^{Nd} 
\inn{f}{h^{\ep}_{N,\nu}}h^{\ep}_{N,\nu} .
\Ee
For these operators one looks for estimates that  are uniform in $\|\fa\|_\infty\le 1$.

The relation between the partial sums $S_R^\cU$ and the operators $\bbE_N$ and $T_N[\cdot, \fa]$ 
is explained in \S\ref{localization} below; see also \cite{gsu, gsu-endpt}. In particular, their uniform boundedness in $F^s_{p,q}$ 
implies that \eqref{uniformSRbound} holds for all strongly admissible enumerations $\cU$. 
The optimal region for the uniform boundedness for $\SE_N$ and $T_N[\cdot,\fa]$ in $F^s_{p,q}$ is given in the next theorem, and depicted in Figure \ref{fig2} below.

\begin{theorem} \label{th3} Let $0<p\leq\infty$, $0<q\leq \infty$ and $s\in \bbR$.

(a) The operators $\SE_N$ admit an extension from $\cS$ into $F^s_{p,q}(\SR^d)$ such that
$$\sup_{N\ge 0} \|\bbE_N\|_{F^s_{p,q}\to F^s_{p,q}} <\infty$$ 
if and only if one of the following five conditions is satisfied:

\medskip

(i) $\;\quad 1<p\leq \infty$, $\quad -1+\frac 1p< s< \frac1p$,  $\quad 0<q\le \infty$,

(ii)  $\quad \frac{d}{d+1} \leq p<1$, $\quad s=1$, $\quad 0<q\le 2$,

(iii)  $\!\quad \frac{d}{d+1} < p\leq 1$, $\quad d(\frac 1p-1)<s<1$, $\quad 0<q\le \infty$,

(iv) $\quad \frac{d}{d+1} <  p\leq 1$, $\quad s=d(\frac 1p-1)$, $\quad 0<q\le \infty$,

(v) $\quad p=\infty$, $\quad s=0$, $\quad 0<q\le \infty$.

\smallskip

\noi (b) If one of the conditions (i)-(v) is satisfied then the operators $T_N[\cdot,\fa]$ are uniformly bounded on $F^s_{p,q}(\bbR^d)$ when $N\geq0$ and $\|\fa\|_{\ell^\infty(\bbZ^d\times\Upsilon)}\le 1$.

\end{theorem}

\begin{figure}[h]\label{uniformbdnessfigure}
 \centering
{\begin{tikzpicture}[scale=2]
\draw[->] (0.03,0.0) -- (2.1,0.0) node[right] {$\frac{1}{p}$};
\draw[->] (0.0,0.03) -- (0.0,1.1) node[above] {$s$};

\draw (1.0,0.03) -- (1.0,-0.03) node [below] {$1$};
\draw (1.5,0.03) -- (1.5,-0.03) node [below] {{\tiny $\;\;\frac{d+1}{d}$}};
\draw (0.03,1.0) -- (-0.03,1.00) node [left] {$1$};

\draw[dotted] (1.0,0.0) -- (1.0,1.0);

\draw[line width=1.1] (0.0,-0.03)--(0.0,-0.97);

\fill (0,0) circle (1pt);
\draw[dashed, thick] (0.02,0.02) -- (0.98,0.98) ;

\draw[thick] (1.03,1.0) -- (1.5, 1.0);
\draw[thick] (1.48,0.96) -- (1.0,0.0);
\fill (1,0) circle (1pt);
\fill (1.5,1) circle (1pt);
\draw (1,1) circle (1pt);
\draw[<-, thin] (1.35,0.6)--(1.45,0.6);
\node [right] at (1.4,0.6) {{\footnotesize $0<q\leq \infty$}};

\draw[<-, thin] (1.25,1.05)--(1.25,1.15)--(1.5,1.15);
\node [ right] at (1.5,1.15) {{\footnotesize $ {0<q\leq 2}$}};

\draw[dashed, thick] (0.02,-0.97) -- (1.0,0.0);
\draw (0,-1) circle (1pt) node [left] {$-1$};

\node [right] at (-0.04,-0.2) {{\tiny $0<q\leq \infty$}};
\draw[->, thin] (0.3,-0.28)--(0.3,-0.35)--(0.05,-0.35);

\end{tikzpicture}
}
\caption{Region for uniform boundedness of $\SE_N$  (hence for the \emph{basic sequence} property) 
in the spaces $F^s_{p,q}(\SR^d)$.}
\label{fig2}
\end{figure}
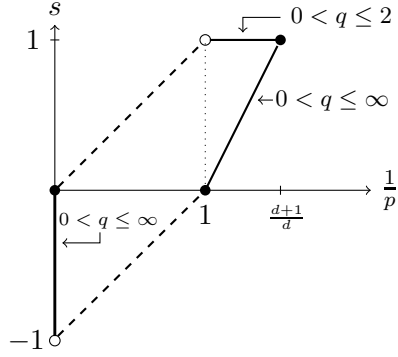

Regarding positive results, the cases (i) and (iii) were established in \cite{gsu}. 
The novel cases appear at the end-point lines in (ii) and (iv), and the special point (v); see Figure \ref{fig2}.

The proof of (ii) will follow from a slightly stronger result which we state next. 
Let  $\eta_0\in C^\infty_c(\SR^d)$  be supported in $\{|\xi|<3/4\}$ with
$\eta_0(\xi)=1$ for $|\xi|\leq 1/4$, and
let $\varPi_N$ be defined by
\Be
\widehat{ \varPi_N f}(\xi)= \eta_0(2^{-N}\xi)\widehat f(\xi).
\label{pinf}
\Ee
 Then we shall
actually prove the following. 

\begin{theorem}\label{expthm} Let ${d}/{(d+1)}\le p<1$ and $0<r\leq\infty$.
Then
\Be \label{TL2bd}
\sup_N    \|\bbE_Nf -\varPi_Nf 
\|_{B^1_{p,r}} \lc  \|f\|_{F^1_{p,2}} .  
\Ee
Moreover, for $\|\fa\|_{\ell^\infty}\leq 1$,
\Be \label{TL1bda}
\sup_N    \|T_N[f,\fa]
\|_{B^1_{p,r}} \lc \|f\|_{F^1_{p,2}}.
\Ee
\end{theorem}
Using the embeddings
$F^s_{p,q} \subset F^s_{p,2}$ for $q\le 2$, and  $B^s_{p,r}\subset F^s_{p,r}\subset F^s_{p,q}$ for $r\leq\min\{p,q\}$,
one deduces the uniform bounds in (ii) of Theorem \ref{th3}.

Likewise, for the end-point cases in (iv) and (v)  
we shall establish the following stronger results.

\begin{theorem}\label{expthm:d/p-d} Let ${d}/{(d+1)}<p\le 1$, $0<r\leq\infty$, and $s=\frac dp -d$.
Then
\Be \label{BT-diff}
\sup_N    \|\bbE_Nf -\varPi_Nf 
\|_{B^s_{p,r}} \lc \|f\|_{F^s_{p,\infty}},
\Ee
and likewise for the operators $T_N[\cdot,\fa]$, 
uniformly in $\|\fa\|_{\ell^\infty}\leq 1$.
\end{theorem}

\begin{theorem}\label{th_s0} 
For every $r>0$, it holds
\Be \label{ENPiN_s0}
\sup_N    \|\bbE_Nf -\varPi_Nf 
\|_{F^0_{\infty,r}} \lc \|f\|_{B^0_{\infty,\infty}},
\Ee
and likewise for the operators $T_N[\cdot,\fa]$, 
uniformly in $\|\fa\|_{\ell^\infty}\leq 1$.
\end{theorem}

Finally, concerning the negative results in Theorem \ref{th3}, the only non-trivial case appears when $s=1$, for which we shall establish the following.

\begin{theorem} \label{schauderF} 
Let $\frac{d}{d+1} \le p< 1$ and $2<q\leq\infty$.
Then,
\[
\|\SE_N \|_{F^1_{p,q}\to F^1_{p,q}}\,\approx \, N^{\frac12-\frac1q}.
\] 
\end{theorem}


\subsection*{\it This paper} In \S\ref{pre} we set the basic notation.
In \S\ref{S_expthm} and \S\ref{S_expthm:d/p-d} we prove, respectively, Theorems \ref{expthm} and \ref{expthm:d/p-d},
except for the special case $p=d/(d+1)$ which is treated in \S\ref{S_expthmendpoint}. Theorem \ref{th_s0}
is shown in \S\ref{S_s0}, and Theorem \ref{schauderF} in \S\ref{S_s=1}.
In \S\ref{S_th3} we gather all these results and complete the proof of Theorem \ref{th3}, explaining as well the meaning
of the extensions of the operators $\SE_N$ to the full spaces $F^s_{p,q}$. 
In \S\ref{S_dense} we study the failure of density for $\Span\sH_d$ in the case $s=1$. 
In \S\ref{localization} and \ref{S_th1} we pass to the operators $S^{\cU}_R$, showing their relation with $\SE_N$ for admissible enumerations, and establishing Theorem \ref{th1}. Finally, \S\ref{S_th_unc} is devoted to unconditionality, and the proof of Theorem \ref{th_unc}. 

\section*{\it Acknowledgements}
The authors would like to thank the Isaac Newton Institute for Mathematical Sciences, Cambridge, for support and hospitality during the program 
 Approximation, Sampling and Compression in Data Science where some work on this paper was undertaken. This work was supported by EPSRC grant no EP/K032208/1. G.G. was supported in part by grants MTM2016-76566-P, MTM2017-83262-C2-2-P and Programa Salvador de Madariaga PRX18/451 from Micinn (Spain), and 
grant 20906/PI/18 from Fundaci\'on S\'eneca (Regi\'on de Murcia, Spain). A.S. was supported in part by National Science Foundation grants 1500162 and 
1764295. 
T.U. was supported in part by 
 Deutsche Forschungsgemeinschaft (DFG), grant 403/2-1.

\section{Preliminaries}\label{pre}

\subsection{\it Besov and Triebel-Lizorkin quasi-norms}
\label{S_qnorm}
Let $s\in\SR$ and $0<p, q\leq \infty$ be given. Throughout the paper we fix a number $A>d/p$ and an integer 
\Be \label{Mcondition} M >A+|s|+2 . 
\Ee 
Consider two functions $\beta_0, \beta \in C^\infty_c(\SRd)$, 
supported in $(-1/2,1/2)^d$, with the properties
 $|\widehat{\beta}_0(\xi)|>0$ if  $|\xi|\leq1$, 
 $|\widehat{\beta}(\xi)|>0$ if $1/8\leq|\xi|\leq1$ and $\beta$ has vanishing
moments up to order $M$, that is
\Be\label{moments-M}
\int_{\SR^d} \beta(x) \;x_1^{m_1}\cdots x_d^{m_d}\,dx =0, \quad\forall\;m_i\in \bbN_0\;\mbox{ with  $\;m_1+\ldots +m_d\leq M$.}
\Ee
The optimal value of $M$ is irrelevant for the purposes of this paper, and \eqref{Mcondition} suffices for our results.
We let  $\beta_k:=2^{kd}\beta(2^k\cdot)$ for each $k\geq1$, and  denote 
\[L_kf=\beta_k*f\]
whenever $f\in \cS'(\SR^d)$. These convolution operators, sometimes called \emph{local means}, can be used to define 
equivalent quasi-norms in the $B^s_{p,q}$ and $F^s_{p,q}$ spaces.
Namely,
\Be
\big\|g\big\|_{B^s_{p,q}}\,\approx\,\Big\|\big\{2^{ks}L_k g\big\}_{k=0}^\infty\Big\|_{\ell^q(L^p)}\, 
\label{Bspq_localmeans}\Ee
and if  $0<p<\infty$, 
\Be
\big\|g\big\|_{F^s_{p,q}}\,\approx\,\Big\|\big\{2^{ks}L_k g\big\}_{k=0}^\infty\Big\|_{L^p(\ell^q)}\, 
\label{Fspq_localmeans}\Ee
 see e.g.  
 \cite[2.5.3 and 2.4.6]{triebel2}. 
For the latter spaces, when $p=\infty$ (and $q<\infty$) one defines instead
\Be
\big\|g\big\|_{F^s_{\infty, q}}\,\approx\,\sup_{n\geq0}\sup_{I\in\sD_n}\Big(\frac1{|I|}\int_I\sum_{k\geq n}\,2^{ksq}|L_kg(x)|^q\,dx\Big)^{1/q}, 
\label{Finfty}
\Ee
see \cite[(12.8)]{FrJa90}, \cite{BuiTa}. Finally, one lets $F^s_{\infty,\infty}=B^s_{\infty,\infty}$.

Next, let $\eta_0\in C^\infty_c(\SRd)$ be supported on $\{\xi: |\xi|<3/8\}$ and such that $\eta_0(\xi)=1$ if $|\xi|\leq1/4$. We define the operators 
\begin{subequations} 
\begin{align}\label{La0def}
\widehat{\La_0 f}(\xi) &=\frac{\eta_0(\xi)}{\widehat \beta_0(\xi)}\widehat f(\xi)\,,
\\
\label{Lakdef}
\widehat{\La_k f}(\xi) &=\frac{\eta_0(2^{-k}\xi) -\eta_0(2^{-k+1}\xi)}{\widehat \beta(2^{-k}\xi)}\widehat f(\xi), \quad k\ge 1,
\end{align}
\end{subequations}
so that  
\Be\label{resofid} f=\sum_{j=0}^\infty L_j \La_j f \Ee  with convergence in $\cS'$. 
Of course, one obtains (the usual) equivalent norms if in \eqref{Bspq_localmeans}, \eqref{Fspq_localmeans} and \eqref{Finfty}
 the operators 
$L_k$ are replaced by $\La_k$.
In particular, if we let $\varPi_N=\sum_{j=0}^NL_j\La_j$, then 
\Be\sup_N\|\varPi_N f\|_{F^s_{p,q}}\lesssim \|f\|_{F^s_{p,q}}.\label{PiN}\Ee

Below we shall be interested in uniformly bounded extensions of the dyadic averaging operators $\SE_N$ defined in \eqref{expect}.
We shall denote 
\[ \bbE_N^\perp= I-\bbE_N\mand \varPi_N^\perp= I-\varPi_N,\]
and write
\Be
\SE_N-\varPi_N=\SE_N\,\varPi_N^\perp\,-\,\SE_N^\perp\,\varPi_N.
\label{EN-PiN}
\Ee
Then, using \eqref{Bspq_localmeans}, we have
\Bea
\big\|\SE_N f-\varPi_N f\big\|_{B^s_{p,r}} & \lesssim &  
\;\Big\|\big\{2^{ks} L_k\SE_N\varPi_N^\perp f\big\}_{k=0}^\infty\Big\|_{\ell^r(L^p)}
\;+ \hskip1cm\label{SENPiN}\\
& & \hskip1cm +
\;\;\Big\|\big\{2^{ks}L_k\SE^\perp_N\varPi_Nf\big\}_{k=0}^\infty\Big\|_{\ell^r(L^p)}.\nonumber
\Eea
Following \cite{gsu, gsu-endpt}, we shall prove Theorems \ref{expthm}, \ref{expthm:d/p-d} and \ref{th_s0} 
using suitable estimates for the functions $L_k\bbE_N L_jg$ and $L_k\bbE_N^\perp L_jg$, for each $j,k\geq0$,
some of which will be new in this paper.

\section{Proof of Theorem \ref{expthm}: The case $p>\frac{d}{d+1}$}\label{S_expthm}

Let $s=1$ and $d/(d+1)<p<1$. For these indices, Theorem \ref{expthm} will be a consequence of the following two results.
The first result is contained in \cite{gsu-endpt} (Propositions 3.1 and 3.4), and was also implicit in \cite{gsu} (proof of inequality (19)).

\begin{proposition}\label{Ps1}
For $\frac d{d+1}<p<1$ and $r>0$, it holds
\Be\label{j>N} \sup_N 
\Big(\sum_{k=0}^\infty 2^{kr} \big\|   L_k  
\,\bbE_N\,\varPi_N^\perp f
\big\|_p^r\Big)^{1/r} \lc  
\|f\|_{B^1_{p,\infty}}.
\Ee 
The same holds if $\SE_N$ is replaced by $T_N[\cdot,\fa]$ with $\|\fa\|_{\ell^\infty}\leq 1$.
\end{proposition}

The second result is new, and it will require a few additional arguments compared to \cite{gsu, gsu-endpt}. 
The conditions on $p$ are also less demanding. Here $h^p=F^0_{p,2}$ is the local Hardy space; see e.g. \cite[2.5.8]{Tr83}.

\begin{proposition}
\label{j<Nprop}
For $\frac{d}{d+2}<p<1$ and $r>0$, it holds
\Be\label{j<N} \sup_N 
\Big(\sum_{k=0}^\infty 2^{kr} \big\|   L_k  
\,\bbE_N^\perp\,\varPi_N f
\big\|_p^r\Big)^{1/r} \lc  
\|\nabla f\|_{h^p}.
\Ee 
The same holds if $\SE_N^\perp$ is replaced by $T_N[\cdot,\fa]$ with $\|\fa\|_{\ell^\infty}\leq 1$.
\end{proposition}

We shall prove  Proposition \ref{j<Nprop} in the next subsections, 
but we indicate now how \eqref{j>N} and \eqref{j<N} imply \eqref{TL2bd}.
Just use the Littlewood-Paley type inequality 
\[
\|\nabla f\|_{h^p}\lesssim\|f\|_{F^1_{p,2}};
\]
(see {\it e.g.} \cite[2.3.8/3]{Tr83}) and the embedding 
 $F^1_{p,2}\hookrightarrow B^1_{p,\infty}$.

\subsection{\it A pointwise estimate}
As in \cite{gsu} 
we shall use the Peetre maximal functions
\Be \label{peetrenontang}\fM_{A,j}^{**} g(x)= \sup_{h\in \SRd}\frac{ |g(x+h)|}{(1+2^j|h|)^A}\,,
\Ee
typically applied to scalar or Hilbert space valued $g\in \cS'(\bbR^d)$  satisfying
\Be\label{specassu}
 \supp \widehat g
 \subset \{\xi: |\xi|\le 2^{j+1}\} .
 \Ee
In \cite{Pe} it was shown that for  $g$  satisfying \eqref{specassu},
\Be \label{peetreLp}
\|\fM_{A,j}^{**} g\|_p\le C_{p,A}
\|g\|_p ,  \text{ $\quad 0<p\leq\infty$,  
 $\quad A>d/p$.}
 \Ee 
 
 In what follows it will  be convenient to use the notation\[
|x|_\infty=\max_{1\leq i \leq d}|x_i|, \quad x=(x_1,\ldots, x_d)\in\SR^d.
\]
The following lemma is a variation of
\cite[(35)]{gsu}. 
The novelty here is that the operator $\SE^\perp_N$ is acting on $\varPi_Nf=\sum_{j\leq N}L_j\La_j f$, rather than in each $L_j\La_j f$ separately.
 
 \begin{lemma}  
 \label{ENperpPiN}
 Let $f\in\cS'(\SR^d)$.
 Then 
 \Be
 \big|\bbE_N^\perp [\varPi_N f](y)\big| \lesssim \!\inf_{|y'-y|_\infty\le  2^{1-N}}\,\fM_{A,N}^{**} (2^{-N}\varPi_N \nabla f)(y'),\quad y\in \SR^d.
\label{bound_infM}
 \Ee
In particular, if $|y-\frac{\mu}{2^{N}}|_\infty\leq 2^{1-N}$, then, for every $p>0$, 
\Be
\label{L1_aux}
|\SE^\perp_N\varPi_N f(y)|\,\lesssim\, \bigg[\mint_{|h|_\infty\leq 2^{2-N}}\big|\fM_{A,N}^{**} (2^{-N}\varPi_N \nabla f)(\tfrac\mu{2^N}+h)\big|^p\,dh\,\bigg]^{1/p}.
\Ee
These bounds also hold if we replace $\bbE_N^\perp$ with $T_N(\cdot, \fa)$ with $\|\fa\|_\infty\le 1$.
 \end{lemma}
 
 \begin{proof}Recall that  $\widehat {\varPi_N f}(\xi)=\eta_0(2^{-N} \xi)\widehat f(\xi)$. Let $\Phi\in\cS$, with $\widehat\Phi(\xi)=1$ when $|\xi|\leq1$, and let $\Phi_N(z)=2^{Nd}\Phi(2^N z) $.
Then $$\varPi_N f=\Phi_N* \varPi_N f.$$
If $I\in\sD_N$ is such that $y\in I$, we have
\begin{align*}
 \big|\SE_N^\perp(\varPi_Nf)(y)\big| &=  \big|\bbE_N[\Phi_N*(\varPi_N f)](y)-\Phi_N*\varPi_Nf(y)\big|\\
& 	\hskip-1.5cm=\Big|\mint_{I} \int_{\SRd} \Phi_N(z)\big[ \varPi_Nf(v-z) -\varPi_Nf(y-z)\big] dz\, dv\Big|
  \\
&\hskip-1.5cm =\Big|\mint_{I} \int_{\SRd}\Phi_N(z)\int_0^1\inn{v-y}{\nabla \varPi_Nf(y+s(v-y)-z)} ds dz\, dv\Big|
  \\
& \hskip-1.5cm \lc  \int_{\bbR^d} |\Phi_N(z)| (1+2^N|z|)^A   \, dz \,
 \; \sup_{\tilde z\in \bbR^d} \frac{|2^{-N}\nabla \varPi_Nf(y'+\tilde z)|}{(1+2^N |\tilde z|)^A}
  \\
& \hskip-1.5cm \le C_A\;  \fM^{**}_{A,N} [2^{-N}\nabla\varPi_N f](y'),  \end{align*} 
	for any $y'$ such that $|y-y'|_\infty\leq 2^{1-N}$. This shows \eqref{bound_infM}. 
	The last assertion in \eqref{L1_aux} follows easily from here.

Finally, if we replace $\bbE_N^\perp$ with $T_N[\cdot, \fa]$, the 
cancellation of $\int_I h_I=0$
implies that, for $w\in I$,
\[
|T_N[\varPi_Nf,a](w)|\leq \Big|\tfrac1{|I|}\int_I h_I(v)\big[\Phi_N*\varPi_Nf(v)-\Phi_N*\varPi_Nf(w)\big]\,dv\Big|.
\]
The rest of the proof is then carried out as above. 
\end{proof}

\subsection{\it Norm estimates}

As in \cite{gsu}, we use the notation \Be
\cU_{N,k}=\Big\{y\in \bbR^d\mid \min_{1\leq i\leq d}\dist (y_i, 2^{-N}\bbZ)\le 2^{-k-1}\Big\}, \quad k>N.
\label{UNj}
\Ee
Roughly speaking, this is the set of points at distance $O(2^{-k})$ from $\Ds\bigcup_{I\in\sD_N}\partial I$.
Note (or recall from \cite[Lemma 2.3.i]{gsu}) that if $k>N$ then
\Be
L_k(\SE_N g)(x)=0, \quad \forall\; x\in \cU_{N,k}^\complement=\SR^d\setminus \cU_{N,k}.
\label{gsu_L1}
\Ee
The next two results will be obtained using Lemma \ref{ENperpPiN}.
 \begin{lemma} \label{LkENPiNklarge}
Let $0<p\leq 1$. Then for every $k>N$ and $\|\fa\|_\infty\le 1$,
 \[ 2^{k}\|L_k \bbE_N^\perp \varPi_N f\|_p + 
  2^{k}\|L_k T_N[\varPi_N f,\fa]\|_p  
 \lc 2^{-(k-N)(\frac 1p-1)} \|\nabla\varPi_N f\|_p.\]
 \end{lemma}
 \begin{proof}
The observation in \eqref{gsu_L1} implies that
\Bea
\big\|L_k\SE_N^\perp\varPi_N f\big\|_p & \lesssim &  
\big\|L_k\SE_N^\perp\varPi_N f\big\|_{L^p(\cU_{N,k}^\complement)}+\big\|L_k\SE_N^\perp\varPi_N f\big\|_{L^p(\cU_{N,k})}\nonumber \\ 
& \lesssim & 
\|L_k\varPi_Nf\|_{L^p(\cU_{N,k}^\complement)}+
\Big[\sum_{\mu\in\SZd}
\|L_k\SE_N^\perp\varPi_N f\big\|_{L^p(\cU_{N,k}\cap I_{N,\mu})}^p
\Big]^\frac1p\!.
\label{kN_aux2}\Eea
Using \eqref{L1_aux} and the fact that $\supp\beta_k(x-\cdot)\subset \mu2^{-N}+O(2^{-N})$ for $x\in I_{N,\mu}$, the last term is controlled by
\Beas
 \Big[\sum_{\mu\in\SZd}|I_{N,\mu}\cap\cU_{N,k}| \;\mint_{|h|_\infty\le 2^{2-N}}\!\! \big|\fM_{N, A}^{**} (2^{-N}\nabla \varPi_Nf)(2^{-N}\mu+h)\big|^p\,dh\Big]^\frac1p &  \\
\quad \lesssim  \,\big[2^{-k}2^{-N(d-1)}\big]^\frac1p\,
2^{\frac {Nd}p}\,\big\| \fM_{N,A}^{**} (2^{-N}\nabla \varPi_Nf)\big\|_p\;\lesssim\;2^{-N}\,2^{\frac{N-k}p}\,\|\nabla \varPi_N f\|_p.  
\Eeas
To estimate the first term on the right hand side of \eqref{kN_aux2},
observe that we can write
\[
\beta_k*\varPi_Nf = 2^{-k}\tilde \beta_k*(\nabla\varPi_Nf),
\]
where $\tilde\beta=(\tilde\beta^1,\ldots,\tilde\beta^d)$ and each $\tilde\beta^i$ is a primitive of $\beta$ in the $x_i$-variable
(hence with vanishing moments up to order $M-1$). Moreover, 
\begin{align*}
\|L_k\varPi_N f\|_{L^p(\cU_{N,k}^\complement)} &\leq 2^{-k} \big\|\tilde\beta_k*[\Phi_N *\nabla \varPi_N f]\big\|_p 
\\&\lc 2^{-k} 2^{-(k-N)(M-A)} \|\nabla \varPi_N f\|_p,
\end{align*}
using in the last step the cancellation of $\tilde\beta_k$; see \cite[Lemma 2.2]{gsu}. Analogous arguments apply for $T_N[\varPi_N f, \fa]$ in place of $\bbE_N^\perp \varPi_N f$.
 \end{proof}

 \begin{lemma} \label{LkENPiNksmall}
  Let $0<p\le 1$. Then, for every $k\le N$, $\|\fa\|_\infty\le 1$,
  \Be 
 2^k \|L_k \bbE_N^\perp \varPi_N f\|_p+
 2^k \|L_k T_N[ \varPi_N f,\fa]\|_p
 \lc 2^{(N-k)(\frac dp-d-2)}  \|\nabla \varPi_N f\|_p.
\Ee
 \end{lemma}
 
 \begin{proof}
   Since   $\int_{I} \bbE_N^\perp[ \varPi_N f](y) \,dy=0$ for $I\in\sD_N,$ 
   we may  write
  \begin{multline*}L_k\big(  \bbE_N^\perp[ \varPi_N f]\big)(x)\\=\sum_{\mu\in \cZ_{k,N}(x)}  \int_{I_{N,\mu} }   \big(\beta_k(x-y) -\beta_k(x-2^{-N}\mu)\big)\,
 \bbE_N^\perp[ \varPi_N f](y)\, dy\,
\end{multline*}
 where
 \Be\label{Lax} 
 \cZ_{k,N}(x) = \{\mu\in \bbZ^d:  |x-2^{-N}\mu|_\infty\le 2^{-N}+2^{-k-1} \}.\Ee 
 Note that  $\card\,\cZ_{k,N}(x)\approx 2^{(N-k)d}$. 
 Now use 
 \[|\beta_k(x-y) -\beta_k(x-2^{-N}\mu)|\lc 2^{kd}\,2^{k-N},\quad \text{ if $y\in I_{N,\mu}$,}\]
 in combination with Lemma \ref{ENperpPiN} to obtain
\begin{align*}
&|L_k\big(  \bbE_N^\perp[ \varPi_N f]\big)(x)| \,\lesssim\,
\\ &\quad
 \,\lc\, 2^{(k-N)(d+1)}  \sum_{\mu\in\cZ_{k,N}(x)}\Big(\mint_{|h|_\infty\leq 2^{2-N}}\!\! \big| \fM_{A,N}^{**} (2^{-N}\nabla \varPi_N f)\big|^p(\tfrac\mu{2^N}+h)\,dh\Big)^\frac1p
 \\
&\quad\lesssim 2^{(k-N)(d+1)} 
 \Big(\sum_{\mu\in\cZ_{k,N} (x)}\mint_{|h|_\infty\leq 2^{2-N}}\!\! \big| \fM_{A,N}^{**} (2^{-N}\nabla \varPi_N f)\big|^p(\tfrac\mu{2^N}+h)\,dh \Big)^\frac1p,
\end{align*} the last step using the embedding $\ell^{1} \hookrightarrow \ell^{1/p}$, since  $p\le 1$.
From this
\begin{align*}
&\|L_k\big(  \bbE_N^\perp[ \varPi_N f]\big)\|_p 
\\& \lc 2^{(k-N)(d+1)} 
\Big[\int_{\SRd}\!\!\sum_{\mu\in\cZ_{k,N}(x)}\mint_{|h|_\infty\leq 2^{2-N}}\!\! \big| \fM_{A,N}^{**} (2^{-N}\nabla \varPi_N f)\big|^p(\tfrac\mu{2^N}+h)\,dh\,dx\Big]^\frac1p 
 \\
 &\lc 2^{(k-N)(d+1)} 
  \Big[\sum_{\mu\in\SZd}2^{-kd}\,\mint_{|h|_\infty\leq 2^{2-N}}\!\! \big| \fM_{A,N}^{**} (2^{-N}\nabla \varPi_N f)\big|^p(\tfrac\mu{2^N}+h)\,dh\Big]^\frac1p\\
  &\lc 2^{(k-N)(d+1)} 
  2^{(N-k)d/p}\,\big\|\fM_{A,N}^{**} (2^{-N}\nabla\varPi_N f)\big\|_p,
\end{align*}
and the assertion follows by
 the Peetre inequality for $\fM^{**}_{A,N}$.
 Analogous arguments apply for $T_N[\varPi_N f, \fa]$ in place of $\bbE_N^\perp \varPi_N f$.
  \end{proof}
  
 \begin{proof}[Proof of Proposition \ref{j<Nprop}]
 Using the
 Lemmas \ref{LkENPiNklarge} and  \ref{LkENPiNksmall}, and noticing that
 we may sum in $k$ since $\frac{d}{d+2}<p<1$, one easily obtains
\[
\Big(\sum_{k=0}^\infty 2^{kr} \big\|   L_k  
\,\bbE_N^\perp\,\varPi_N f
\big\|_p^r\Big)^{1/r}\lesssim\,\|\varPi_N\nabla f\|_p.
\]
The last quantity can be estimated further, applying to $g=\nabla f$  the inequality
 \Be\label{partialsum} \|\varPi_N g\|_p\leq \Big\|\sup_{N\geq0}|\varPi_Ng|\Big\|_p \lc \|g\|_{h^p}\approx \|g\|_{F^0_{p,2}},
 \Ee
 which follows for  example using  
the standard maximal function characterization of the $h^p$ norm. 
This proves \eqref{j<N}.
 The proof for the operators $T_N[\cdot,a]$ is exactly analogous.
   \end{proof}

\section{Proof of Theorem \ref{expthm:d/p-d}: the case $s=d/p-d$}
\label{S_expthm:d/p-d}

Let $s=d/p-d$ and $d/(d+1)<p\leq1$ (we will take up the endpoint case $p=d/(d+1)$, when $s=1$ in \S\ref{S_expthmendpoint}). For these indices, Theorem \ref{expthm:d/p-d} will be a consequence of the following two results.
The first result was already established in \cite{gsu-endpt} (Propositions 3.2 and 3.3), using the same type of analysis as in \cite{gsu}.
The inequality is slightly stronger than needed due to $F^s_{p,\infty}\hookrightarrow B^s_{p,\infty}$.

\begin{proposition}\label{firsttwosums}
Let  $\frac{d}{d+1}<p\le 1$ and $r>0$. Then
\Be\Big(\sum_{k=0}^\infty 2^{k(d/p-d)r}\Big \|
L_k \bbE_N^\perp \varPi_Nf\Big\|_p^r\Big)^{1/r} \lc \|f\|_{B^{\frac dp-d}_{p,\infty}}.
\label{subeq1}\Ee
The same holds if $\SE_N^\perp$ is replaced by $T_N[\cdot,\fa]$ with $\|\fa\|_{\ell^\infty}\leq 1$.
\end{proposition}

The second proposition is new, and its proof will require several additional refinements compared to the arguments given in \cite{gsu}. 
\begin{proposition}\label{Finftyprop}
Let $\frac{d-1}{d}<p\le 1$ and $r>0$. Then,  
\begin{align}
\label{jkgeN-ple1}
&\Big(\sum_{k=N+1}^\infty 2^{k(d/p-d)r}\Big \|\sum_{j>N}
L_k \bbE_N L_j \La_j f\Big\|_p^r\Big)^{1/r} \lc \|f\|_{F^{\frac dp-d}_{p,\infty}}
\\
\label{jgeNkleN-ple1}
&\Big(\sum_{k=0}^N 2^{k(d/p-d)r}\Big \|\sum_{j>N}
L_k \bbE_N L_j \La_j f\Big\|_p^r\Big)^{1/r}
\lc \|f\|_{F^{\frac dp-d}_{p,\infty}}.
\end{align}
The same holds if $\SE_N$ is replaced by $T_N[\cdot,a]$ with $\|a\|_{\ell^\infty}\leq 1$.
\end{proposition}

\subsection{\it Notation and  observations on dyadic cubes}\label{omegacubesect}
Recall that every dyadic cube $I$ is contained in a unique parent cube 
of double side length. Also each dyadic cube has $2^d$ children cubes of half side length. It will be useful to single out one of the children cubes according to the following definition.

\begin{definition} Let $I$ be a dyadic cube. We denote by $\om(I)$ the unique child of $I$ with the property that its closure contains the center of the parent cube of $I$.
\end{definition}

We need some further notation (taken from \cite{gsu}).
For each dyadic cube $I\in\sD_N$, we denote by $\sD_N(I)$ the set of all its neighboring $2^{-N}$-cubes, that is,  $I'\in\sD_N$ with $\bar{I}\cap \bar{I'}\not=\emptyset$. 
Likewise, if $\ell>N$   we   denote by 
$\sD_\ell[\partial I]$ the set of all $J\in \sD_{\ell}$ such that 
$\bar{J}\cap\partial I\not=\emptyset$.

\begin{lemma} \label{omegacubelemma}
(i) Let $J\in \sD_\ell[\partial I]$. Then 
$$2^{-\ell-1}\le \dist(x,\partial I)_\infty\le 2^{-\ell} \text{ for all $x\in \omega(J)$.}$$

(ii) Let $I\in \sD_N$, let $\ell_1, \ell_2>N$ and consider two distinct  cubes
$J_1\in \sD_{\ell_1}[\partial I]$, 
$J_2\in \sD_{\ell_2}[\partial I]$. Then $\omega(J_1)$, $\omega(J_2)$ have disjoint interiors.
\end{lemma} 

\begin{proof} The upper bound in (i) is true for all $x\in J$, by definition of $\sD_\ell[\partial I]$ and the lower bound follows  from the definition of $\omega(J)$ since  the parent cube of $J$  is contained in $I$ or one of its neighbors of equal side length. To see  (ii)  first observe that $J_1$, $J_2$ are disjoint if $\ell_1=\ell_2$ (and hence $\om(J_1)$ and $\omega(J_2)$ are disjoint). If $\ell_1\neq \ell_2$ then (ii) follows from (i).
\end{proof}

\subsection{\it Proof of Proposition \ref{Finftyprop}}
We make a preliminary observation about maximal functions.
If  $g$ is continuous, for each $j\geq0$ we let\[
\fM^*_jg(x):=\sup_{|h|_\infty\leq 2^{-j}}|g(x+h)|.
\]
Then, if a cube $J\in \sD_{j+1}$ has center $c_J$, we have
\Be\label{LinftyJ}
\sup_{x\in J}|g(x)|\leq\,\inf_{|h|_\infty\leq 2^{-j-1}}\fM^*_jg(c_J+h)\leq
\Big[
\mint_{\om(J)} |\fM^*_jg|^p\,\Big]^{\frac1p}.
\Ee

\begin{proof}[Proof of \eqref{jkgeN-ple1}]
Let $j, k>N$. By \eqref{gsu_L1}, 
\[
L_k\bbE_N[L_j\La_j f](x)=0\quad \mbox{ if $x\in\cU_{N,k}^\complement$}.
\] 
 Moreover, by 
 \cite[Lemma 2.3 (ii)]{gsu}
 we have
 \Be \label{ENjmax} |\bbE_N (L_j\Lambda_j f)(y)|\lc 2^{(N-j)d} 
 \sum_{I\in \sD_N} \bbone_I(y) \sum_{J\in \sD_{j+1}[\partial I]} \|\Lambda_j f\|_{L^\infty(J)}.
 \Ee
 Let  $x\in \cU_{N,k}\cap I$, for some 
 $I\in\sD_N$. Then $\supp\beta_k(x-\cdot)\subset \bigcup_{I'\in\sD_N(I)}I'$, and therefore \eqref{ENjmax} implies
 \begin{align*}
|L_k\bbE_N[L_j \La_jf](x)|& \leq  \int|\beta_k(x-y)|\,\big|\SE_N(L_j\La_j f)(y)\big|\,dy\\
& \lc    2^{(N-j)d} \sum_{I'\in\sD_N(I)}\sum_{J\in\sD_{j+1}[\partial I']}
\|\La_j f\|_{L^\infty(J)}\,\|\beta_k\|_{1}.
\end{align*}
Using the inequality in \eqref{LinftyJ}, this in turn implies (since $p\leq 1$)
\begin{align}
A_k(x)^p & :=  \Big[\sum_{j>N}|L_k\SE_NL_j\La_jf(x)|\Big]^p \label{intwJ}
 \\
& \lesssim  \sum_{j>N}
2^{(N-j)dp} \!\!\!
\sum_{I'\in\sD_N(I)}\sum_{J\in\sD_{j+1}[\partial I']}
\mint_{\om(J)}|\fM^*_j(\La_jf)|^p
\nonumber
\\
 & \lesssim  
2^{Ndp}\sum_{I'\in\sD_N(I)}\sum_{j>N}\sum_{J\in\sD_{j+1}[\partial I']}
\int_{\om(J)}\sup_{\ell>N}|2^{\ell(\frac dp-d)}\fM^*_\ell(\La_\ell f)|^p.\nonumber
\end{align}
By Lemma \ref{omegacubelemma} the sets  $\om(J)$ for $J\in \sD[\partial I']$
are disjoint. Also since $\#\sD_N[I]=2^d$ 
we obtain
\[
A_k(x)^p\, \lesssim\, 2^{Ndp}\,\int_{I^{**}}\sup_{\ell>N}\big|2^{\ell(\frac dp-d)}\fM^*_\ell(\La_\ell f)\big|^p,
\]
with $I^{**}$ the five-fold dilation of $I$  with respect to $c_I$.
Thus, if we  write  \Be
G= \sup_{\ell>N}|2^{\ell(\frac dp-d)}\fM^*_\ell(\La_\ell f)|,
\label{Gdef}
\Ee 
we obtain
\Beas 
\| A_k\|_p^p & \lesssim & 2^{Ndp}\,\sum_{I\in\sD_N}|I\cap\cU_{N,k}|\,\int_{I^{**}}|G|^p\\
& \lesssim & 2^{N(d-\frac{d-1}p)p}\, 2^{-k}\,\|G\|^p_{L^p(\SR^d)},
\Eeas
and therefore
\Be
2^{k(\frac dp-d)}\| A_k\|_p  \lesssim   2^{-(k-N)(d-\frac{d-1}p)}\,\|G\|_p.
\label{AkG}
\Ee
When $p>(d-1)/d$ we can sum in $k>N$, and hence the left hand side of \eqref{jgeNkleN-ple1}
is controlled by $\|G\|_p$.
Now, Peetre's inequalities imply that 
\[
\fM^*_\ell[\La_\ell f](x)\lesssim \fM^{**}_{A,\ell}[\La_\ell f](x)\lesssim M_\sigma[\La_\ell f](x),\]
for $\sigma=d/A$, where 
\[
M_\sigma g(x)=\sup_{R>0}\Big[\mint_{B_R(x)}|g|^\sigma \Big]^{1/\sigma}.
\]
Thus, 
\Bea
\|G\|_p & \lesssim & \Big\|\sup_{\ell >N} 2^{\ell (\frac dp-d)}
M_\sigma[\La_\ell  f]\Big\|_p \leq  \Big\|
M_\sigma\Big[\sup_{\ell > N} 2^{\ell (\frac dp-d)}|\La_\ell  f|\Big]\Big\|_p\,\nonumber \\
& \lesssim &    \Big\|
\sup_{\ell > N} 2^{\ell (\frac dp-d)}|\La_\ell  f|\Big\|_p\,
\lc\, \|f\|_{F^{\frac dp-d}_{p,\infty}}\label{HLmaximal}
\Eea
using the boundedness of the Hardy-Littlewood maximal operator, since $\sigma=A/d<p$.  
This finishes the proof of \eqref{jkgeN-ple1}, for the stated version involving $\bbE_N$. The analogous version for $T_N[\cdot,\fa]$ follows similarly, by replacing 
\eqref{ENjmax} with the corresponding version for the $T_N$, as in \cite{gsu}.
\end{proof}

\begin{proof}[Proof of \eqref{jgeNkleN-ple1}]
Let $j>N$ and $k\leq N$. Again, we shall follow the proof of \cite[(26)]{gsu},
applying the same changes as in the previous subsection. Namely, let
$$\widetilde f_j= (\La_j f)\bbone_{\cU_{N,j}}$$
and note that $\bbE_N [L_j \La_j f]=\bbE_N[L_j \widetilde f_j]$; see \cite[Lemma 2.3]{gsu}.
Then we write
 \Be
	L_k\bbE_N[L_j f] 
   = L_k(\bbE_N[L_j \widetilde f_j] -L_j \widetilde f_j) + L_kL_j \widetilde f_j.
	\label{jN_eq}\Ee
	As in \cite{gsu}, the last term is harmless since 
	\[
	\|L_kL_j\widetilde f_j\|_p\lesssim 2^{-(M-A)|j-k|}\,\big\|\fM^{**}_{A,j}(\La_j f)(x)\big\|_p\lesssim 2^{-(M-A)|j-k|}\,\|\La_j f\|_p,
	\]
by \cite[Lemma 2.2]{gsu}. Thus, assuming $r\le p$ (as we may), and using the $r$-triangle inequality, 
\begin{align*}&\Big(\sum_{k=0}^N 2^{k(\frac dp-d)r} \|\sum_{j>N} L_k L_j\widetilde f_j\|_p^r\Big)^{1/r}\,\lesssim\,
	\\&
	\Big(\sum_{k=0}^N 
	\sum_{j>N}
	2^{(k-j)(\frac dp-d+(M-A))r} 2^{j(\frac dp-d)r} \|\La_j f\|_p
	^r\Big)^{1/r} \lc \|f\|_{B^{\frac dp-d}_{p,\infty}}\,.
	\end{align*}
	
Hence it remains to prove
\Be\label{jgeNkleNperp-ple1}
\Big(\sum_{k=0}^N 2^{k(d/p-d)r}\Big \|\sum_{j>N}
L_k \bbE_N^\perp  L_j \widetilde f_j\Big\|_p^r\Big)^{1/r}
\lc \|f\|_{F^{\frac dp-d}_{p,\infty}}.
\Ee
Following \cite{gsu}, and letting  $\cZ_{k,N}(x)$ be as in \eqref{Lax}, we write
\Beas
\cA_{j,k}(x) & := & |L_k\big(  \bbE_N^\perp[ L_j\widetilde f_j]\big)(x)| \\
& \leq &  \sum_{\mu\in\cZ_{k,N}(x)}\Big| \int_{I_{N,\mu} }   \big(\beta_k(x-y) -\beta_k(x-2^{-N}\mu)\big)\,
 \SE_N^\perp[ L_j\widetilde f_j](y)\, dy\,\Big|\\
& \lc & 2^{kd} 2^{k-N} \sum_{\mu\in\cZ_{k,N}(x)} \int_{I_{N,\mu} } \Big( |\SE_N[ L_j\widetilde f_j](y)|+|L_j\widetilde f_j(y)|\Big)\, dy.
\Eeas
In \cite[p. 1332]{gsu}, the terms corresponding to the two summands in the integral are estimated differently, but produce essentially the same outcome, namely
\Be
\cA_{j,k}(x)\, \lesssim \,2^{k-N}2^{(k-j)d}\Big(\sum_{\mu\in\cZ_{k,N}(x)}	
\sum_{J\in\sD_{j+1}[\partial I_{N,\mu}]}\|\La_jf\|_{L^\infty(J)}^p\Big)^\frac1p,
\label{A_aux1}
\Ee
see \cite[(41)]{gsu}. At this point we argue as in the previous subsection. That is, we use  \eqref{LinftyJ} to have 
\Be
\|\La_j f\|_{L^\infty(J)}\leq \Big[
\mint_{\om(J)} \fM^*_j[\La_j f](y)^p\,dy\Big]^{\frac1p},
\label{intwJ2}
\Ee
and conclude that
\begin{align}
\cA_k(x)^p & :=  \Big[\sum_{j>N}|L_k\SE^\perp _NL_j\widetilde f_j(x)|\Big]^p \label{defofcAk}\\
& \lesssim  \sum_{j>N}
2^{(k-N)p} 2^{(k-j)dp}
\sum_{\mu\in\cZ_{k,N}(x)}
\sum_{J\in\sD_{j+1}[\partial I_{N,\mu}]}
\mint_{\om(J)}|\fM^*_j(\La_jf)|^p
\notag\\
 & \lesssim 
2^{(k-N)p} 2^{kdp}\sum_{\mu\in\cZ_{k,N}(x)}
\int_{I^{**}_{N,\mu}} |G|^p 
\notag
\end{align}
with $G$ as in \eqref{Gdef}, and using the disjointness of the sets $\om(J)$ as before.
Thus, integrating the above expression 
\Beas 
\| \cA_k\|_p^p & \lesssim & 2^{(k-N)p} 2^{kdp}\,\sum_{\mu\in\SZ^d}
2^{-kd}
\int_{I^{**}_{N,\mu}} |G|^p\\
& \lesssim & 2^{(k-N)p} 2^{k(d-\frac dp)p} \,
\int_{\SR^d} |G|^p\,
\Eeas
and therefore
\Be
2^{k(\frac dp-d)}\| \cA_k\|_p  \lesssim   2^{k-N}\,\|G\|_p.
\label{cAkG}
\Ee
Therefore, one can sum in $k\leq N$, and obtain the desired expression in \eqref{jgeNkleNperp-ple1} using the estimate for $\|G\|_p$ in \eqref{HLmaximal}.
This finishes the proof of \eqref{jgeNkleN-ple1}. The corresponding version for $T_N$ is proved similarly (notice that in 
\eqref{jN_eq} the analysis of the last summand becomes unnecessary, due to the additional cancellation of $T_N$).
\end{proof}

\section{Proof of Theorem \ref{expthm}: The case $p=\frac{d}{d+1}$}\label{S_expthmendpoint}

The end-point case $p=d/(d+1)$ and $s=1$, was excluded from the previous proofs because of the 
restrictions imposed in Propositions \ref{Ps1} and \ref{firsttwosums}. However, one can use instead Propositions \ref{Finftyprop} and \ref{j<Nprop}, 
which are valid 
at this endpoint. Namely, they imply the inequalities
\begin{align}
\label{jkgeN-ple1-s=1}
&\sup_N\,\Big(\sum_{k=0}^\infty 2^{kr}\Big \|\sum_{j>N}
L_k \bbE_N L_j \La_j f\Big\|_p^r\Big)^{1/r} \lc \|f\|_{F^{1}_{p,\infty}},
\quad p=\frac{d}{d+1},
\end{align}
and 
\Be\label{j<N-s=1}
\sup_N 
\Big(\sum_{k=0}^\infty 2^{kr} \Big\|   L_k  
\bbE_N^\perp\varPi_N f
\Big\|_{p}^r\Big)^{1/r} \lc  \|f\|_{F^1_{p,2}} 
\quad p=\frac{d}{d+1}.
\Ee 
Then, the result stated in  Theorem \ref{expthm} follows using additionally the embedding $F^1_{p,2}\hookrightarrow F^1_{p,\infty}$ in \eqref{jkgeN-ple1-s=1}. 
The same argument applies to $T_N[\cdot,\fa]$ with $\|\fa\|_\infty \le 1$ if we use the corresponding versions of Propositions \ref{Finftyprop} and \ref{j<Nprop}.
\qed

\section{Proof of Theorem \ref{th_s0}: the case $s=0$ and $p=\infty$}
\label{S_s0}

In view of \eqref{EN-PiN}, it suffices to prove the following.
\begin{proposition}
\label{P_s0}
Let $r>0$. Then
\Be
\big\|\SE_N^\perp\varPi_N f\big\|_{F^0_{\infty,r}}\,+\,\big\|\SE_N\varPi_N^\perp f\big\|_{F^0_{\infty,r}}  \, \lesssim \,  \|f\|_{B^0_{\infty,\infty}}.
\label{s0a}
\Ee
\end{proposition}

One part of the estimates will be derived from
the following inequalities, proved in \cite[(36a), (37a)]{gsu-endpt}:
\Bea
\Big(\sum_{k\leq N}\big\|L_k\SE^\perp_N\varPi_Nf\big\|_\infty^r\Big)^{1/r} & \lesssim & \|f\|_{B^0_{\infty,\infty}}
\label{36a}\\
\Big(\sum_{k\leq N}\big\|L_k\SE_N\Pi^\perp_Nf\big\|_\infty^r\Big)^{1/r} & \lesssim & \|f\|_{B^0_{\infty,\infty}}
\label{37a}.
\Eea
We remark that these same inequalities with $\sum_{k\leq N}$ replaced by  $\sum_{k> N}$ are only true if $r=\infty$.
This necessitates 
the use of $F_{\infty,r}^0$-norms on the left  hand side of \eqref{s0a}

To establish the proposition, let $f\in B^0_{\infty,\infty}$ be such that $\|f\|_{B^0_{\infty,\infty}}=1$.
We shall prove separately each of the two inequalities. 

\subsection{\it Proof $\big\|\SE_N^\perp\varPi_N f\big\|_{F^0_{\infty,r}}\lesssim 1$}
By \eqref{Finfty} we can write 
\begin{align*}
\big\|\SE_N^\perp\varPi_N f\big\|_{F^0_{\infty,r}}^r\,&=\,\sup_{\ell\geq0}\sup_{I\in\sD_\ell} A^{(\ell)}_I
\\
\quad \text{where }
A^{(\ell)}_I&
:=\mint_I\sum_{k\geq\ell}|L_k\SE^\perp_N\varPi_N f|^r.
\end{align*}
If $0\leq \ell\leq N$, then, for each $I\in\sD_\ell$,
\Be
A^{(\ell)}_I\leq\sum_{k=\ell}^N\big\|L_k\SE^\perp_N\varPi_Nf\|_\infty^r\,+\,
\mint_I\sum_{k> N}\big|L_k\SE^\perp_N \varPi_Nf\big|^r=:A^{(\ell)}_{I,1}+A^{(\ell)}_{I,2}.
\label{Aell_aux}
\Ee
By \eqref{36a} we have $A^{(\ell)}_{I,1}\lesssim 1$. For the second term, one can split $I$ into $ 2^{(\ell-N)d}$ disjoint cubes $J\in\sD_{N}$, so that
\[
A^{(\ell)}_{I,2}\,\leq\, \sup_{{J\in\sD_{N}}\atop{J\subset I}} A^{(N)}_{J} .
\]
Thus, it suffices to show that
\Be
\sup_{\ell\geq N} A^{(\ell)}_{I}\,\lesssim\, 1.
\label{goal2}
\Ee
Let $\ell\geq N$ and $I\in\sD_\ell$. Then \eqref{gsu_L1} gives $L_k(\SE_N g)\equiv0$ in $\cU_{N,k}^\complement$, if $k\geq \ell$, so
\[
 A^{(\ell)}_{I}=\frac1{|I|}\sum_{k\geq \ell}\int_{I\cap \cU_{N,k}}\big|L_k\SE^\perp_N\varPi_Nf\big|^r.
\]
We shall show that
\Be
\label{st1}
\big|L_k\SE^\perp_N\varPi_Nf(x)\big|\lesssim \|f\|_{B^0_{\infty,\infty}}=1,\quad \text{for }x\in I.
\Ee
This inequality combined with $|I\cap \cU_{N,k}|\approx 2^{-(d-1)\ell}2^{-k}$ will establish the result, since
\[
 A^{(\ell)}_{I}\lesssim \frac1{|I|}\sum_{k\geq \ell}|I\cap \cU_{N,k}|\lesssim \sum_{k\geq \ell}2^{\ell-k}\lesssim1 .
\]
It remains to show \eqref{st1}. Let $Q\in\sD_N$ be such that $I\subset Q$. By \eqref{L1_aux} 
\Be
\big|\SE_N^\perp\varPi_N f(y)\big|\lesssim\mint_{Q^{**}}\fM^{**}_{A,N}\big(2^{-N}\varPi_N\nabla f\big),\quad \text{for }y\in I^*.
\label{las1}
\Ee
Now,
\Bea
\Big\|\fM^{**}_{A,N}\big(2^{-N}\varPi_N\nabla f\big)\Big\|_\infty & \leq & \big\|2^{-N}\varPi_N\nabla f\big\|_\infty\, \lesssim \, \sum_{j\leq N} 2^{j-N}\|\La_j f\|_\infty
\nonumber
\\
& \lesssim & \sup_{m\geq 0}\|\La_m f\|_\infty\lesssim\|f\|_{B^0_{\infty,\infty}}= 1.
\label{las2}
\Eea
So, if $x\in I$, then $\supp\beta_k(x-\cdot)\subset x+O(2^{-k})\subset I^*$, and using \eqref{las1} and \eqref{las2} one deduces \eqref{st1}.
This completes the proof of \eqref{goal2}.

\subsection{\it Proof $\big\|\SE_N\varPi_N^\perp f\big\|_{F^0_{\infty,r}}\lesssim 1$}
We now must bound 
\begin{align*}
\big\|\SE_N\varPi_N^\perp f\big\|_{F^0_{\infty,r}}^r&=\,\sup_{\ell\geq0}\sup_{I\in\sD_\ell} B^{(\ell)}_I
\\
\text{where }
B^{(\ell)}_I
&=\mint_I\sum_{k\geq\ell}|L_k\SE_N\varPi_N^\perp f|^r.
\end{align*}
The cases $0\leq \ell\leq N$ are handled with the same argument as in \eqref{Aell_aux},
this time using the inequality \eqref{37a}. 
If $\ell\geq N$ and  $I\in\sD_\ell$ we shall use
\[
 B^{(\ell)}_{I}\leq \frac1{|I|}\sum_{k\geq \ell}\,|I\cap \cU_{N,k}|\;\big\|L_k\SE_N\varPi_N^\perp f\big\|_\infty^r,
\]
so that it will suffice to show
\Be
\label{st1b}
\big\|L_k\SE_N\varPi_N^\perp f\big\|_\infty\lesssim \|f\|_{B^0_{\infty,\infty}}=1.
\Ee
If $Q\in\sD_N$, then  \eqref{ENjmax} and the argument in \eqref{intwJ} (with $\omega(J)$ as in \S\ref{omegacubesect}) implies
\Bea
\big|\SE_N L_j\La_j f(y)\big| & \lesssim &  2^{(N-j)d}\sum_{J\in\sD_{j+1}(\partial Q)} \|\La_j f\|_{L^\infty(J)}\label{2Nj} \\
& \lesssim &  2^{(N-j)d}\sum_{J\in\sD_{j+1}(\partial Q)} \mint_{\om(J)}\fM^{*}_{j}\big(\La_j f\big),\quad y\in Q,\nonumber
\Eea
using in the last step \eqref{LinftyJ}. So, summing up in $j>N$ and using the disjointness properties of the sets $\omega(J)$ we obtain
\Beas
\sum_{j>N}\big|\SE_N L_j\La_j f(y)\big|  & \lesssim &    2^{Nd} \int_{Q^{**}}\sup_{m\geq 0}\fM^{*}_{m}\big(\La_m f\big)\\
& \lesssim & \sup_{m\geq 0}\|\La_m f\|_\infty\lesssim\|f\|_{B^0_{\infty,\infty}}= 1.
\Eeas
Finally, taking the convolution with $\beta_k$ one easily deduces \eqref{st1b}.
This completes the proof of \eqref{s0a}. The corresponding version for the $T_N[\cdot,\fa]$ is proved similarly. Thus the proof of Proposition \ref{P_s0} is complete,  and so is the proof of  Theorem \ref{th_s0}.

\begin{remark}
\label{rem_conv_s0}
The above proof also shows that, if $f\in B^0_{\infty,\infty}$, and $N$ is fixed, then the series
$\sum_{j=0}^\infty \SE_N(L_j\La_j f)$ 
converges in the norm of $F^0_{\infty, r}$, for all $r>0$. This is a consequence of the crude bound
\[
\textstyle\|\sum_{j=J_1}^{J_2}L_j\La_j f\|_{F^0_{\infty,r}}\lesssim_N\, 2^{-J_1}\,\|f\|_{F^0_{\infty,\infty}},
\]
which can be obtained from \eqref{2Nj}; see also \cite[Remark 4.5]{gsu-endpt}. 

\end{remark}

 \section{Proof of Theorem \ref{schauderF}: Necessary conditions for $s=1$}
\label{S_s=1}

Here we show the assertion in Theorem \ref{schauderF}, which corresponds to 
the optimality of the range of $q$ stated in (ii) of Theorem \ref{th3}.
More precisely, we establish the following.

\begin{theorem} \label{th_NC_s1} 
Let $d/(d+1)\leq p<1$ and $2<q\leq\infty$. Then 
\Be
\|\SE_N \|_{F^1_{p,q}\to F^1_{p,q}}\,\approx \, N^{\frac12-\frac1q}.
\label{normENs1}
\Ee
Moreover, for every $N\geq1$ there exists $g_N\in C^\infty_c((0,1)^d)$ such that 
\Be
\|g_N\|_{F^1_{p,q}}\leq1 \mand
\|\SE_N(g_N) \|_{F^1_{p,\infty}}\,\gtrsim \, N^{\frac12-\frac1q}\;.\label{limsup_PN}
\Ee
\end{theorem}

\subsection{\it Proof of Theorem \ref{th_NC_s1}: upper bounds}
Let $s=1$. Using Proposition \ref{Ps1} when $d/(d+1)<p<1$, or Proposition \ref{Finftyprop} when $p=d/(d+1)$, one has
the inequality
\[
\,\Big\|\big\{ 2^{kq}\sum_{j>N}
L_k \bbE_N L_j \La_j f\big\}_{k=0}^\infty\Big\|_{L^p(\ell^q)} \lc \|f\|_{F^{1}_{p,\infty}}.
\]
On the other hand, the proof of Proposition \ref{j<Nprop} gives 
\[
\,\Big\|\big\{ 2^{kq}
L_k \bbE^\perp_N \varPi_N f\big\}_{k=0}^\infty\Big\|_{L^p(\ell^q)} \lc \|\varPi_N f\|_{F^{1}_{p,2}},
\]
and by H\"older's inequality one has
\[
\|\varPi_N f\|_{F^{1}_{p,2}}\lesssim \Big\|\big(\sum_{j=0}^{N}|2^j\La_jf|^2\big)^\frac12\Big\|_p\lesssim N^{\frac12-\frac1q}\,\|f\|_{F^1_{p,q}}.
\]
Combining the above inequalities one obtains $\|\SE_N \|_{F^1_{p,q}\to F^1_{p,q}}\lesssim N^{\frac12-\frac1q}$.
\ProofEnd

\begin{remark}
If $1<s<1/p$
, the upper bound becomes exponential:
\[
\big\|\SE_N\big\|_{F^s_{p,q}\to F^s_{p,q}}\,\lesssim\,2^{(s-1)N},
\]  
for  $(d-1)/d<p<1$. This is a consequence of the simpler estimates for $\SE_N-\varPi_N:B^s_{p,\infty}\to B^s_{p,r}$
shown in \cite[Propositions 3.1 through 3.4]{gsu-endpt}. From \cite{gsu} we have also corresponding matching lower bounds, see the discussion  in \S\ref{S_th3} below.
\end{remark}

\subsection{\it Proof of Theorem \ref{th_NC_s1}: lower bounds}

To make the notation simpler, the counterexample is first presented in the 1-dimensional case,
and later extended to $\SR^d$ with a tensor product argument.

 \subsubsection{The case $d=1$} Consider, for $s>0$ and $\La\subset\SN$, a Weierstrass-type function
	\Be
	f(x)\,=\,\Big(\sum_{j\in \La} \frac{a_j}{2^{sj}}\, e^{2\pi i 2^j x}\Big)\,\psi(x), \quad x\in\SR,
	\label{fWe}\Ee
	with $\psi\in C^\infty_c(0,1)$, and say $\psi=1$ on $[1/4,3/4]$.
	These functions satisfy \Be
	\|f\|_{F^s_{p,q}(\SR)}\approx\Big\|(\sum_{k=0}^\infty|2^{ks} \beta_k*f|^q)^\frac1q\Big\|_p \lesssim \big(\sum_{j\in\La}|a_j|^q\big)^{1/q}.
	\label{alq}
	\Ee
This can for instance be proved from Hardy's inequalities and the following lemma

\begin{lemma}\label{bkj}
Let $\beta_k$ be as in $\S2$, and $\psi_j(x)=e^{2\pi i 2^jx}\psi(x)$. Then
\[
|\beta_k*\psi_j(x)|\lesssim 2^{-|j-k|M},\quad x\in\SR, \quad j,k=0,1,2,\ldots
\]
\end{lemma}	
\begin{proof}
If $k>j$, using that $\beta$ has $M$-vanishing moments,
\Bea
|\beta_k*\psi_j(x)|& = & \Big|\int_\SR\beta(y)\big[\psi_j(x-2^{-k}y)-\sum_{m=0}^{M-1}\psi_j^{(m)}(x)(-2^{-k}y)^m\big]\,dy\Big|\nonumber\\
& \lesssim & 2^{(j-k)M},\label{aux_kj1}
\Eea
since $\|\psi^{(M)}_j\|_\infty\lesssim 2^{jM}$. If $k\leq j$, then Fourier inversion gives, for any $M_1>1$, 
\begin{align*}
|\beta_k*\psi_j(x)| &=  \Big|\int_\SR\hat\beta(\xi/2^k)\,\hat\psi(\xi-2^j)\,e^{2\pi i x\xi}\,d\xi\Big|
\\&\, \lesssim \, \int_\SR\frac{d\xi}{(1+\frac{|\xi|}{2^k})^{M_1}\,(1+|\xi-2^j|)^{M_1}}\\
& \lesssim  \int_{|\xi-2^j|>2^{j-1}}\frac{2^{-jM_1}\,d\xi}{(1+\frac{|\xi|}{2^k})^{M_1}}+
 \int_{|\xi-2^j|\leq 2^{j-1}}\frac{2^{(k-j)M_1}\,d\xi}{(1+|\xi-2^j|)^{M_1}}\\
& \lesssim   2^k 2^{-jM_1}\,+\,2^{(k-j)M_1}
\end{align*}
and we get the  $O(2^{-|j-k|M})$ bound  if we choose $M_1\ge M$.
\end{proof}
	
We now let 	$s=1$ and $$\fZ_N=\{j\in \bbN: N/4\le j\le  N/2\},$$ and consider a randomized version of \eqref{fWe}, namely
\Be
\label{fNt}
f_N(x,t) = \sum_{j\in \fZ_N} \frac{r_j(t)}{2^j} \,
e^{2\pi i 2^j x}\, \psi(x),\Ee
	where $r_j:[0,1]\to \{-1,1\}$ is the sequence of Rademacher functions.
Then, by \eqref{alq},
\[
 \sup_{t\in[0,1]} \|f_N(\cdot, t)\|_{F^1_{p,q}} \lc N^{1/q}.
\]
Below  we shall show that
\Be\label{fWel} \Big(\int_0^1 \big\| \bbE_N[f_N(\cdot,t)]\big\|^p_{F^1_{p,\infty}}\,dt\Big)^{1/p} \ge c N^{1/2}.
 \Ee
The above inequality will be a consequence of the estimate
\Be
\Big(\int_0^1 \big\| 2^{N}\,\beta_N*(\bbE_N[f_N(\cdot,t)])\big\|^p_p\,dt\Big)^{1/p} \ge c N^{1/2},
\label{fWel2}
\Ee
where $\beta_N=2^{Nd}\beta(2^N\cdot)$ and $\beta$ is a suitable test function satisfying the conditions in \S2.1.
Thus for some $t_0\in[0,1]$  the function 
\Be
g_N\,=\, N^{-1/q}\,f_N(\cdot,t_0)
\label{gNfN}
\Ee will satisfy 
\Be
\|\SE_N g_N\|_{F^1_{p,q}(\SR)}\gtrsim \big\| 2^{N}\,\beta_N*(\bbE_N g_N)\big\|_{L^p(\SR)}\gtrsim N^{\frac12-\frac1q},
\label{fWel2bis}
\Ee
and in particular
 \Be
\big\|\bbE_N\big\|_{F^1_{p,q}\to F^1_{p,q}}\,\gtrsim\, N^{\frac12-\frac1q}.
\label{ENnorm}\Ee

By Fubini's theorem and Khintchine's inequality the expression in \eqref{fWel2} is equivalent to
 \Be
2^N\,\Big\| \big(\sum_{j\in\fZ_N}\big|2^{-j}\,\beta_N*(\bbE_N\psi_j)\big|^2\big)^\frac12\Big\|_{L^p(\SR)} \ge c N^{1/2}.
\label{fWel3}
\Ee
If the operator $\SE_N$ is omitted in the left hand side, then this quantity becomes uniformly bounded by Lemma \ref{bkj}, so \eqref{fWel3} is also equivalent to
 \Be
2^N\,\Big\| \big(\sum_{j\in\fZ_N}\big|2^{-j}\,\beta_N*(\bbE_N^\perp\psi_j)\big|^2\big)^\frac12\Big\|_{L^p(\SR)} \ge c N^{1/2}.
\label{fWel4}
\Ee
Below we fix $\beta$ such that $\supp \beta=(-1/8,1/8)$, and denote its primitive by $B(x)=\int_{-\infty}^x \beta(u)du$, which also belongs to $C^\infty_c(-1/8,1/8)$. The following lemma is similar to \cite[Lemma 6.4]{gsu-endpt}, but we include its proof below for completeness.

\begin{lemma}
\label{BN}
Let $\mu\in\SZ$ and let $\tI_{N,\mu}=[\frac{\mu}{2^N}, \frac{\mu+1/8}{2^N}]$. Then
\Be
\beta_N*(\SE^\perp_N\psi_j)(x)=-2^{-N}\psi_j'(\tfrac\mu{2^N})\,B(2^Nx-\mu)\,+\,O\big(2^{2(j-N)}\big),\quad x\in\tI_{N,\mu}.
\label{asy}
\Ee
Moreover, if $\frac\mu{2^N}\in[\frac14,\frac34]$, then $\psi_j'(\tfrac\mu{2^N})=2\pi i 2^j e^{2\pi i 2^j x}$.
\end{lemma}

Assuming the lemma, the $p$th power of the left hand side of \eqref{fWel4} can be bounded from below by 
\begin{align*}
 & \sum_{{\mu\in\SZ}\atop{\frac14\leq\frac\mu{2^N}\leq\frac34}}\int_{\tI_{N,\mu}}
\Big(\sum_{j\in\La}\big|2^{N-j}\,\beta_N*(\bbE_N^\perp\psi_j)\big|^2\Big)^\frac p2\,dx\\
& \geq  (2\pi)^p\sum_{{\mu\in\SZ}\atop{\frac14\leq\frac\mu{2^N}\leq\frac34}}\int_{\tI_{N,\mu}}
\Big(\sum_{j\in\La}\big|B(2^Nx-\mu)\big|^2\Big)^\frac p2-c\big(\sum_{j\in\La}|2^{j-N}|^2\big)^\frac p2
\,dx\\
& \gtrsim  N^\frac p2\,\sum_{{\mu\in\SZ}\atop{\frac14\leq\frac\mu{2^N}\leq\frac34}}\int_{\tI_{N,\mu}}
\big|B(2^Nx-\mu)\big|^p\,dx\;\;\;-\;c'\,2^{-Np/2}\\
& \gtrsim  N^\frac p2\,\int_{0}^{1/8}
\big|B(u)\big|^p\,du\;-\;c'\,2^{-Np/2}\;\gtrsim\; N^\frac p2,
\end{align*}
using in the last step that $\beta$ (hence $B$) is not identically null in $(0,1/8)$.  This finishes the proof modulo Lemma \ref{BN}.

\begin{proof}[Proof of Lemma \ref{BN}]
For simplicity we write $I^+=I_{N,\mu}$ and $I^-=I_{N,\mu-1}$. If $x\in\tI_{N,\mu}$, then $\supp \beta_N(x-\cdot)\subset I^+\cup I^-$, and thus
\Be
\beta_N*(\SE^\perp_N\psi_j)(x)=\sum_{\pm}\int_{I^\pm}\beta_N(x-y)\Big[\psi_j(y)-\mint_{I^\pm}\psi_j\Big]\,dy.
\label{aux_l1}\Ee
Now, if $y\in I^\pm$,  using the linear Taylor's expansion of $\psi_j$  around $y$ and the bound $\|\psi_j^{''}\|_{\infty}\lesssim 2^{2j}$, the inner bracketed expression becomes
\Beas
\mint_{I^\pm}(\psi_j(y)-\psi_j(w))dw & = & \mint_{I^\pm}\psi'_j(y)(y-w)\,dw + O(2^{2(j-N)})\\
 & = & \psi'_j(\tfrac\mu{2^N})\,\mint_{I^\pm}(y-w)\,dw + O(2^{2(j-N)})\\
& = & \psi'_j(\tfrac\mu{2^N})\,(y-c_{I^\pm})\, + O(2^{2(j-N)})\\
& = & \psi'_j(\tfrac\mu{2^N})\,(y-\tfrac\mu{2^N})\, -\psi'_j(\tfrac\mu{2^N})\,\tfrac{\pm1}{2^{N+1}}+ O(2^{2(j-N)}).
\Eeas
Putting these quantities into \eqref{aux_l1}, and using the support and the moment condition of $\beta$, we are left with
\Be
\beta_N*(\SE^\perp_N\psi_j)(x)=-\sum_{\pm}\psi'_j(\tfrac\mu{2^N})\,\tfrac{\pm1}{2^{N+1}}\int_{I_\pm}\beta_N(x-y)\,dy\;+ O(2^{2(j-N)}).
\label{aux2_l1}\Ee
Now, an elementary computation using the primitive, $B(u)$, of $\beta(u)$ shows that the two integrals substructed above can be written as
\[
\int_{\frac{\mu}{2^N}}^{\frac{\mu+1}{2^N}}\,-\,\int_{\frac{\mu-1}{2^N}}^{\frac{\mu}{2^N}} \beta_N(x-y)dy\quad =
\int_{\mu}^{\mu+1}-\!\int_{\mu-1}^{\mu} \beta(2^Nx-u)\,du=2B(2^Nx-\mu),
\]
since $B(2^Nx-\mu\pm1)=0$ by the support condition.
Thus, placing this expression into \eqref{aux2_l1} implies the asserted identity \eqref{asy}.
\end{proof}

\subsubsection{The $d$-dimensional case}\label{S_ExFd}
Consider $G_N(x_1,x')=g_N(x_1)\chi(x')$, where $g_N$ is the 1-dimensional function in \eqref{gNfN},
and $\chi\in C^\infty_c(0,1)^{d-1}$ with $\chi\equiv1$ in $[\frac18,\frac78]^{d-1}$.
We shall show that 
\Be
\|G_N\|_{F^1_{p,q}(\SR^d)}\lesssim 1
\label{GN1}
\Ee
and
\Be
\big\|\SE_NG_N\big\|_{F^1_{p,q}(\SR^d)}\gtrsim N^{\frac12-\frac1q}.
\label{GN2}
\Ee

To do so, in the definition of the $F^s_{p,q}$-quasinorms we shall use suitable test functions of tensor product type; see also \cite[\S5.1]{gsu-endpt} for a similar argument.
Namely, for a fixed $M\in\SN$  we consider a non-negative even function $\phi_0\in C^\infty_c(-\tfrac18, \tfrac 18)$ such that  
$\phi_0^{(2M)}(t)>0$ for all $t$ in some interval $[-2\e,2\eps]$.
Since $\widehat\phi_0(0)=\int\phi_0>0$, dilating if necessary we may also assume that $\widehat \phi_0\neq 0$ on $[-1,1]$.
Let $\varphi_0 \in C^\infty_c((-\tfrac 18, \tfrac 18)^{d-1})$ be 
such that $\widehat \varphi_0\neq 0$ on $[-1,1]^{d-1}$ and $\widehat \varphi_0(0)=1$.
For $M\ge 1$, let \[ \phi(t):= (\tfrac{d}{dt})^{2M} \phi_0(t), \qquad
\varphi (x_2,\dots, x_d):= \big(
\tfrac{\partial^2}{\partial x_2^2} +\dots+
\tfrac{\partial^2}{\partial x_d^2} \big)^M \varphi_0(x').
\] 
Then, we define 
\Be\Psi(x):= 
\Delta^M[ \phi_0\otimes \varphi_0] (x)= \phi(x_1) \varphi_0(x') + \phi_0(x_1) \varphi(x').
\label{Psitensor}
\Ee
Clearly, $$\int_{\SR^d} \Psi (y) y_1^{m_1}\cdots y_d^{m_d} 
dy=0, \quad\mbox{when }m_1+\ldots+m_d< 2M.$$
Finally, let $\Psi_0=\phi_0\otimes\varphi_0$, and $\Psi_k(x)=2^{kd}\Psi(2^kx)$, $k\geq 1$.
Then, if we choose $M$ sufficiently large we have 
 \Be\label{FpqPsi}
 \|f\|_{F^{s}_{p,q}}
 \,\approx\,
 \Big\| \big\{ 2^{ks}\Psi_k*f\big\}_{k\geq 0}\Big\|_{L^p(\ell^q)},\quad \forall\;f\in 
F^{s}_{p,q}(\SR^d);
 \Ee
see e.g. \cite[2.4.6]{triebel2}. Observe that, for $k\geq 1$ we can write
\Be
\Psi_k=\phi_k\otimes \varphi_{0,k}\,+\,\phi_{0,k}\otimes\varphi_k,
\label{ox}
\Ee
where we denote 
\[
 \phi_k(x_1)=2^k\phi(2^kx_1),\quad \varphi_k(x')=2^{(d-1)k}\varphi(2^kx'),
\]
and likewise for $\phi_{0,k}$ and $\varphi_{0,k}$.

With this notation the inequality in \eqref{GN1} is easily proved as follows.
From \eqref{ox} and $\|g_N\|_\infty,\|\chi\|_\infty\lesssim1$ one obtains
 \[
|\Psi_k*G_N|(x_1,x')\lesssim |\phi_k*g_N|(x_1)+|\varphi_k*\chi|(x'),\quad k\geq 1,\]
and a similar (simpler) expression when $k=0$. Therefore \eqref{FpqPsi} and the compact support of the
involved functions imply 
\[
\|G_N\|_{F^1_{p,q}(\SR^d)}\lesssim \|g_N\|_{F^1_{p,q}(\SR)}\,+\, \|\chi\|_{F^1_{p,q}(\SR^{d-1})}\lesssim1.
\]
In order to prove \eqref{GN2}, we let $\bbE_N^{(1)} $ and $\bbE_N ^{(d-1)}$ be the dyadic averaging operators on $\bbR$ and $\bbR^{d-1}$, respectively. For $N\geq1$, we observe that
\Be\label{1Dredux}
\Psi_N* (\bbE_N G_N)(x_1,x') =  \phi_N * (\bbE_N^{(1)} g_{N}) (x_1),
\text{ for $x'\in(\tfrac14,\tfrac34)^{d-1}$.}
\Ee 
Indeed, for such $x'$ one has
\begin{align*}
&\varphi_{0,N}* (\bbE_N^{(d-1)} \chi)(x')= \int \varphi_0 (y')dy'=1,
\\
& \varphi_N* (\bbE_N^{(d-1)} \chi)(x')= \int \varphi (y')dy'=0,
\end{align*}
due to the support properties of $\varphi_{0,N}(x'-\cdot)$ and $\varphi_N(x'-\cdot)$.
Therefore, \eqref{FpqPsi} and \eqref{1Dredux} imply that
\[
\|\SE_NG_N\|_{F^1_{p,q}(\SR^d)}\gtrsim \big\|2^N\phi_N*\SE^{(1)}_Ng_N\big\|_{L^{p}(\SR)}\,\gtrsim N^{\frac12-\frac1q},
\]
the last inequality due to \eqref{fWel2bis}.
This proves \eqref{GN2}, and concludes the proof of 
\[ \big\|\bbE_N\big\|_{F^1_{p,q}(\SR^d)\to F^1_{p,q}(\SR^d)}\,\gtrsim\, N^{\frac12-\frac1q},\quad 2<q\leq \infty.
 \] 
\ProofEnd

\section{Boundedness of the  dyadic averaging operators and the proof of Theorem \ref{th3}}
\label{S_th3}

We now gather the results from the previous sections to complete the proof of Theorem \ref{th3}.
In \S\ref{extension} we  first explain how the extension of $\SE_N$, from $\cS$ into $F^s_{p,q}$, should be defined (which is not obvious for all  cases). We discuss sufficient  conditions for uniform boundedness in \S\ref{S_suff} using theorems in previous chapters, and necessary conditions in \S\ref{S_nc}.
The proofs of some of the more tedious details about definability are given in \S\ref{App1}. The proof of necessary  conditions for the individual boundedness of the $\bbE_N$ is given  in \S\ref{App2}. In \S\ref{addendum} we  include a discussion when the characteristic function of a bounded interval can be defined as a linear functional on $F^s_{p,q}$.

\subsection{\it Extension of the operators $\SE_N$ to the space $F^s_{p,q}$}\label{extension}
Let $(s,p,q)$ be as in (i)-(v). For a distribution $f\in F^s_{p,q}$ we define
\Be
\SE_Nf:=\sum_{j=0}^\infty\SE_N(L_j\La_j f).
\label{EN_def}
\Ee
We claim that this series always converges in the $F^s_{p,q}$-norm
(actually, in all the $B^s_{p,r}$-norms, for $r>0$). When $\max\{d/p-d,1/p-1\}<s<1/p$ 
this fact was already justified in \cite[Remarks 3.5 and 4.5]{gsu-endpt}. 
When $s=d/p-d$, one can reach  the same conclusion with a slight modification in the proof of Proposition \ref{Finftyprop}.
We present the details in Lemma \ref{L_App1} below.
When $s=0$ and $p=\infty$, the convergence holds in all $F^0_{\infty,r}$-norms, for $r>0$, by Remark \ref{rem_conv_s0}.

We also remark that when $f\in F^s_{p,q}$ is locally integrable with polynomial growth 
then the above extension coincides with the usual operator, that is
\[
\sum_{j\geq 0}\SE_N(L_j\La_j f)=\sum_{I\in\sD_N}\Big(\mint_If\Big)\bbone_I;
\]
see Lemma \ref{L_App2} below.

\subsection{\it Sufficient conditions in Theorem \ref{th3} }\label{S_suff}
The uniform boundedness of $\SE_N$ in $F^s_{p,q}$ in the cases (i) and (iii) was established in \cite{gsu}.
In the cases (ii), (iv) and (v) it is a consequence of the Theorems \ref{expthm}, \ref{expthm:d/p-d} and \ref{th_s0}, and elementary embeddings; see the discussion 
following \eqref{pinf} in section 1.

\subsection{\it Necessary conditions in Theorem \ref{th3}}
\label{S_nc}
We first identify the range of exponents for the continuity of an individual operator $\SE_N$.

\begin{definition}\label{Adef}
Let $\fA$ be the set of all  $(s,p,q)$ for which one of  
 the following three conditions (i), (ii) or (iii)  hold:
\begin{equation}\label{Adefeq}
\begin{cases}
&\text{
\sline (i) $\;\;\quad\max\{\frac dp-d,\,\frac1p-1\}\,<\,s\,<\,1/p$, $\quad 0<p,q\leq \infty$}
\\
&\text{
\sline (ii) $\;\quad s=\frac dp-d$, $\quad 0<p\leq1$, $\quad 0<q\leq \infty$
}
\\
&\text{
\sline (iii) $\quad s=0$, $\quad p=\infty$, $\quad 0<q\leq \infty$.
}
\end{cases}
\end{equation}
\end{definition}

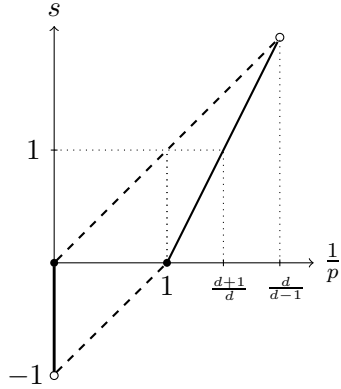
\begin{figure}[ht]\label{uniformbdnessfigure2}
 \centering
{\begin{tikzpicture}[scale=1.5]
\draw[->] (0.03,0.0) -- (2.3,0.0) node[right] {$\frac{1}{p}$};
\draw[->] (0.0,0.03) -- (0.0,2.1) node[above] {$s$};

\draw (1.0,0.03) -- (1.0,-0.03) node [below] {$1$};
\draw (1.5,0.03) -- (1.5,-0.03) node [below] {{\tiny $\;\;\frac{d+1}{d}$}};
\draw (2,0.03) -- (2,-0.03) node [below] {{\tiny $\;\;\frac{d}{d-1}$}};

\draw (0.03,1.0) -- (-0.03,1.00) node [left] {$1$};

\draw[dotted] (1.0,0.0) -- (1.0,1.0);
\draw[line width=1.1] (0.0,-0.03)--(0.0,-0.97);

\fill (0,0) circle (1pt);
\draw[dashed, thick] (0.02,0.02) -- (1.98,1.98) ;

\draw[dotted] (2,0)--(2,1.98);
\draw[dotted] (1,0)--(1,1);
\draw[dotted] (0,1)-- (1.5, 1.0)--(1.5,0);

\draw[thick] (1.98,1.96) -- (1.0,0.0);
\fill (1,0) circle (1pt);
\draw (2,2) circle (1pt);


\draw[dashed, thick] (0.02,-0.97) -- (1.0,0.0);
\draw (0,-1) circle (1pt) node [left] {$-1$};

\end{tikzpicture}
}
\caption{Range of exponents for the continuity of the individual operators $\SE_N$ in $F^s_{p,q}(\SR^d)$.}
\label{fig2b}
\end{figure}

\begin{proposition}
\label{P_e}
Let $N\in\SN_0$ be fixed. Suppose that
\Be
\big\|\SE_N\psi\big\|_{F^s_{p,q}}\leq c_N\|\psi\|_{F^s_{p,q}},\quad\forall\,\psi\in\cS(\SR^d).
\label{e0}
\Ee
Then necessarily $(s,p,q)\in \fA$. 
\end{proposition} 

The proof of the proposition is based on the fact that $\bbone_{[0,1)^d}$ must belong to both $F^s_{p,q}$ and its dual space. 
We present the details in \S\ref{App2}  below.

We now turn to the existence of uniformly bounded extensions of the operators $\SE_N$ in the above region. 
This is equivalent to the uniform boundedness of the
numbers \[
{\rm Op}_\cS(\SE_N, F^s_{p,q}):=\sup\Big\{\big\|\SE_Nf\|_{F^s_{p,q}}\mid f\in\cS,\; \|f\|_{F^s_{p,q}}\leq 1\Big\}, 
 \]
when $N=0,1,2,\ldots$ An example given in \cite[Proposition 4.2]{gsu} shows that, 
\[
{\rm Op_\cS}(\SE_N, F^s_{p,q})\,\gtrsim\, {\rm Op_\cS}(\SE_N, B^s_{p,\infty})\,\gtrsim\, 2^{(s-1)N},
\]
when $1<s<1/p$ and $(d-1)/d<p<1$. So the condition $s\leq 1$ is necessary. 
When $s=1$, Theorem \ref{th_NC_s1} shows that $0<q\leq 2$ is also necessary.  
This comprises all the cases considered in Theorem \ref{th3}, and completes the proof of all the assertions. 
Moreover, it also gives the following.

\begin{corollary}
\label{C_e}
Let $N\in\SN_0$ be fixed, and $(s,p,q)\in\fA$. Then the (extended) operator $\SE_N$, as in \eqref{EN_def}, satisfies
\[
{\rm Op}_\cS(\SE_N,F^s_{p,q})\,\approx \,\big\|\SE_N\big\|_{F^s_{p,q}\to F^s_{p,q}}\,\approx\,\begin{cases}
 2^{(s-1)N} & \mbox{if   $1<s<1/p$}\\
N^{1/2-1/q} & \mbox{if   $s=1$, $q\ge 2$}\\
1 & \mbox{otherwise.}
\end{cases}
\]
\end{corollary}

\subsection{{\it Convergence of $\sum_j\bbE_N(L_j\La_j  f)$ when 
$f\in F^{d/p-d}_{p,\infty}$}}\label{App1}


\begin{lemma}\label{L_App1}
Let $\frac{d-1}d<p\leq 1$, $s=d/p-d$ and $r>0$.  If $N\geq 0$ and $g\in F^s_{p,\infty}(\SR^d)$, then
\[
\lim_{J_1\to \infty}\sup_{J_2\ge J_1}\Big\|\SE_N\big(\sum_{j=J_1}^{J_2}L_j\La_jg\big)\Big\|_{B^s_{p,r}}=0.
\]
In particular, the series
\[
\SE_Ng:=\sum_{j=0}^\infty\SE_N(L_j\La_j g)
\]
converges in $F^s_{p,q}$ for all $0< q \leq \infty$.
\end{lemma}

\begin{proof}
Pick a non-negative function $\zeta\in\cS(\SR^d)$ such that
\[
\zeta\geq 1 \quad\mbox{on $[-5,5]^d$},\mand \supp\widehat\zeta\subset\{|\xi|\leq 1/8\}.
\]
For each cube $Q=\{x\mid |x-x_0|_\infty\leq \dt\}$, let
\[
\zeta_Q(x):=\zeta\Big(\frac{x-x_0}\dt\Big),
\]
so that $\zeta_Q\geq 1$ in $Q^{**}$ (the 5-fold dilate of $Q$), and $\widehat\zeta_Q$ has support in $\{|\xi|\leq(8\dt)^{-1}\}$.
Finally, for $j\geq N+1$, we define (with $\cU_{N,j}$ as in \eqref{UNj})
\[
\zeta_{j, N}(x):=\sum_{{Q\in\sD_{j+1}}\atop{Q\subset \cU_{N,j}}}\zeta_Q(x).
\]
This function satisfies the properties
\Be
|\zeta_{j,N}(x)|\leq\frac{c_M}{\Ds\big(1+2^j\min_{1\leq i\leq d}\dist(x_i,2^{-N}\SZ)\big)^M},
\label{zeta_decay}
\Ee
for each $M>0$, and
\Be
\supp\widehat\zeta_{j,N}\subset \big\{|\xi|\leq 2^{j-3}\big\}.
\label{zeta_spectrum}
\Ee
Let $f=\sum_{j=J_1}^{J_2}L_j\La_jg$, and assume that $J_1\geq L+3$ for some fixed $L> N$.
Then, $f\in F^s_{p,\infty}$ and $\La_jf\equiv 0$ unless $j\geq L$. We now follow the proof of Proposition \ref{Finftyprop},
with the following modification. If $J\in \sD_{j+1}$ is such that $J\subset\cU_{N,j}$, then for all $y\in J$,
\[
\fM_j^*(\La_jf)(y)=\sup_{|h|_\infty\leq 2^{-j}}|\La_j f(y+h)|\leq \fM^*_j\big(\zeta_{j,N}\La_jf\big)(y).
\]
One can use this estimate in \eqref{intwJ} (or in \eqref{intwJ2}), so the same arguments
which lead to \eqref{AkG} (or to \eqref{cAkG}) can be applied with the function $\La_\ell f$ replaced by $\zeta_{\ell,N}\La_\ell f$.
That is there exists $\gamma>0$ such that for $A_k$, $\cA_k$ as in \eqref{intwJ}, 
\eqref{defofcAk}, resp., 
\[
2^{k(\frac dp-d)}(\|A_k\|_p+\|\cA_k\|_p)
\lesssim 2^{-|k-N|\ga}\,\|G_L\|_p,
\]
with
\[
G_L(x)=\sup_{\ell\geq L} \;2^{\ell(\frac dp-d)}\fM^*_\ell\big(\zeta_{\ell,N}\La_\ell f\big).
\]
Since the spectrum of $\zeta_{\ell,N}\La_\ell f$ is contained in $\{|\xi|\leq 2^\ell\}$, one can use Peetre's inequality
and deduce as in \eqref{HLmaximal} that
\[
\|G_L\|_p\lesssim \Big\|\sup_{\ell\geq L}2^{\ell(\frac dp-d)}\zeta_{\ell,N}\,|\La_\ell f|\Big\|_p
\lesssim \Big\|\zeta^*_{L,N}\,\sup_{\ell\geq 0}2^{\ell(\frac dp-d)}\,|L_\ell\La_\ell g|\Big\|_p,\]
where
\[
\zeta^*_{L,N}(x)=\sup_{\ell\geq L} \zeta_{\ell,N}(x)\lesssim \Big(1+2^L\min_{1\leq i\leq d}(x_i, 2^{-N}\SZ)\Big)^{-M}.
\]
Observe that 
\[
\lim_{L\to\infty}\zeta^*_{L,N}(x)=0,\quad \forall\; x\in\bigcup_{n=1}^\infty\cU_{N,n}^\complement,
\]
and therefore at almost every $x\in\SR^d$. So, the assumption $g\in F^s_{p,\infty}$ and the Dominated Convergence Theorem imply that
\[
\lim_{L\to \infty}\|G_L\|_p=0.
\]
\end{proof}

\begin{lemma}\label{L_App2}
Let $(s,p,q)$ be as in (i)-(v) in Theorem \ref{th3}. Let $f\in F^s_{p,q}$ 
be locally integrable with polynomial growth, that is, 
$f(x)/(1+|x|)^M\in L^1(\SR^d)$ for some $M\geq0$. 
Then 
\Be
\sum_{j\geq 0}\SE_N(L_j\La_j f)=\sum_{I\in\sD_N}\Big(\mint_If\Big)\bbone_I,
\label{EN_dist}
\Ee
in the sense of tempered distributions.
\end{lemma}
\begin{proof}
In this lemma we restrict the notation \[
\SE_Ng(x):=\sum_{I\in\sD_N}\Big(\mint_Ig\Big)\bbone_I(x),\quad x\in\SR^d,
\] 
only to locally integrable functions $g$ with polynomial growth.
In particular, $\SE_Ng$ is another such function, hence a tempered distribution.
We write $\widetilde\SE_Nf$ for the distribution on the left hand side of \eqref{EN_dist}. 
If $\psi\in \cS(\SR^d)$, then 
\Be
\big(\widetilde\SE_Nf,\psi\big)=\sum_{j=0}^\infty\big(\SE_N(L_j\La_jf),\psi\big)= \sum_{j=0}^\infty\int_{\SR^d}(L_j\La_jf)\,\SE_N\psi.
\label{tEN}
\Ee
The family of operators $\{\varPi_n=\sum_{j=0}^nL_j\La_j\}_{n\geq0}$ is a smooth approximation of the identity, and therefore
\[
\varPi_nf\to f \quad \mbox{in $L^1(\SR^d, (1+|x|)^{-M}dx)$,}
\]
by the condition on $f$. Therefore, using that $|\SE_N\psi(x)|\lesssim (1+|x|)^{-M}$, we can pass the sum inside the integral in the last expression of \eqref{tEN},
and continuing with Fubini's theorem obtain
\[
\big(\widetilde\SE_Nf,\psi\big)= \int_{\SR^d}(\sum_{j=0}^\infty L_j\La_jf)\,\SE_N\psi= \int_{\SR^d}f\,\SE_N\psi=\int_{\SR^d}\SE_Nf\,\psi.
\]
Hence $\widetilde\SE_Nf$ coincides with $\SE_Nf$ as distributions.
\end{proof}

\begin{remark} One can extend the domain  of $\bbE_N$ further, dropping the polynomial growth assumption in Lemma \ref{L_App2} if in the resolution of the identity \eqref{resofid} we replace the operator $L_N$ by suitable compactly supported convolution kernels.
Indeed there are, for $\eps>0$, $M<\infty$, $C^\infty$ functions 
$\phi$, $\widetilde \phi$, $\psi$, $\widetilde \psi$ supported in $\{|x|\le \eps \}$ such  that $\int \phi=1$, $\int \widetilde \phi=1$ and 
$1-\widehat \phi$, $1-\widehat{\widetilde \phi}$,  $\widehat \psi$, $\widehat{\widetilde \psi}$ all vanish of order $M$ at $0$, and such that  for distributions $f$
\Be\label{localresofid} L_0\widetilde L_0 f+ \sum_{k=1}^\infty L_k\widetilde L_k f=f
\Ee in the sense of  distributions; here $L_0$, $\widetilde L_0$ are the convolution operators with convolution kernels $\phi$, $\widetilde \phi$, resp., and
for $k\ge 1$, $L_k$ and $\widetilde L_k$ are the convolution operators with convolution kernels $2^{(k-1)d} \psi(2^{k-1}\cdot)$, $2^{(k-1)d} \widetilde \psi(2^{k-1}\cdot)$, resp. The resolution in the form \eqref{localresofid} is perhaps not widely known; a proof can be found in \cite[Lemma 2.1]{st}, together with some extensions.
For us the use of the nonlocal operators $\Lambda_N$ has the advantage that we may apply the Peetre maximal inequalities in a straightforward way.
\end{remark}

\subsection{{\it Proof of Proposition \ref{P_e}}}\label{App2}


Since $\SE_N[\psi](x)= \SE_0[\psi(2^{-N}\cdot)](2^Nx)$, we may assume that $N=0$.
Then \eqref{e0} takes the form
\Be
\big\|\SE_0\psi\big\|_{F^s_{p,q}}\leq c_N\|\psi\|_{F^s_{p,q}},\quad\forall\,f\in\cS(\SR^d).
\label{e1}
\Ee
Let $\psi\in C^\infty_c((0,1)^d)$ such that $\int\psi=1$. Then $\SE_0(\psi)=\bbone_{[0,1)^d}$, and \eqref{e1} implies
\Be
\bbone_{[0,1)^d}\in F^s_{p,q}(\SR^d).
\label{e2}
\Ee
The validity of this property is well-known. If $0<p<\infty$, then \eqref{e2} holds iff $s<1/p$.
If $p=\infty$, then \eqref{e2} holds iff $s\leq 0$. See e.g. \cite[Proposition 2.50]{triebel4}.
This gives the required upper bounds on $s$.

We turn to the lower bounds for the exponent $s$. Consider the classes of test functions
\Beas
\cF_1& = & \Span\Big\{\eta(x_1)\cdots\eta(x_d)\mid \eta\in C^\infty_c(-1,1) \;\mbox{ {\small odd}}\,\Big\}\\
\cF_2& = & \Big\{\eta(x_1)\chi(x')\mid \eta\in C^\infty_c(-1,1) \;\mbox{ {\small odd,} }\,\mbox{{\small $\chi\in C^\infty_c(0,1)^{d-1}$ with $\int\chi=1$}}\Big\}.
\Eeas
We first show that, if $f\in\cF_i$, $i=1,2$, then
\Be
\SE_0(f)\,=\,\Big(\int_{[0,1]^d}f\Big)\,\cdot\,h_i,
\label{e3}
\Ee
for some fixed functions $h_1, h_2\in\Span\{\bbone_I\mid I\in\sD_0\}$.

Let $f\in\cF_1$. It suffices to show \eqref{e3} for $f(x)=\eta(x_1)\cdots\eta(x_d)$, with $\eta\in C^\infty_c(-1,1)$ odd.
Given $\e=(\e_1,\ldots,\e_d)\in\{0,1\}^d$ we denote
\[
Q_\e=[0,1)^d-\e \mand \sign(\e)=\prod_{i=1}^d(-1)^{\e_i}.
\]
We claim that \eqref{e3} holds with 
\Be
h_1=\sum_{\e\in\{0,1\}^d}\sign(\e)\,\bbone_{Q_\e}.
\label{e4}
\Ee
Indeed, for such $f$ we have
\[
\SE_0(f)=\sum_{\e\in\{0,1\}^d}\,\mint_{Q_\e}f\,
\cdot\,\bbone_{Q_\e},
\]
and since $\eta$ is odd
\[
\mint_{Q_\e}f=\prod_{i=1}^d\int_{[0,1)-\e_i}\!\!\!\eta(x_i)\,dx_i=
\prod_{i=1}^d(-1)^{\e_i}\int_0^1\eta(x_i)\,dx_i=\sign(\e)\int_{[0,1)^d}f.
\]
Thus, \eqref{e3} follows. 

Similarly, let $f\in\cF_2$, and denote $Q_0=[0,1)^d$ and $Q_1=Q_0-(1,0,\ldots,0)$. Then
\[
\SE_0(f)=\mint_{Q_0} f\,\cdot\,\bbone_{Q_0}\,+\,\mint_{Q_1} f\,\cdot\,\bbone_{Q_1}=\Big(\int_{Q_0}f\Big)\,\big(\bbone_{Q_0}-\bbone_{Q_1}\big),
\]
and hence \eqref{e3} holds with $h_2=\bbone_{Q_0}-\bbone_{Q_1}$. 
This completes the proof of \eqref{e3}, and reduces the proof of Proposition \ref{P_e} to the following result.

\begin{lemma}
\label{L_e1}
Let $I=(0,1)^d$. Suppose that 
\Be
\Big|\,\int_I\, f\;\Big|\;\leq\; C\,\|f\|_{F^s_{p,q}},\quad \forall\;f\in\cF_1\cup\cF_2.
\label{e5}
\Ee
Then one of the following two conditions must hold

\sline (a) $\;\quad s\,>\,\max\{\frac dp-d,\,\frac1p-1\}$, $\quad 0<p,q\leq \infty$

\sline (b) $\;\quad s=\frac dp-d$, $\quad 0<p\leq1$, $\;\quad 0<q\leq \infty$.
\end{lemma}
\begin{proof}
Let $\eta\in C^\infty_c(-1/2,1/2)$, odd and such that $\int_0^1\eta=1$.
Define
\[
g_j(x_1,\ldots, x_d):=2^{jd}\eta(2^jx_1)\cdots\eta(2^jx_d),\quad j\geq 1.
\]
Observe that $g_j\in \cF_1$ and $
\int_I g_j(x)\,dx\,=\,1.$ 
On the other hand, a standard computation shows that
\Be
\|g_j\|_{F^s_{p,q}(\SR^d)}\,\lesssim\, 2^{-j(\frac dp-d-s)}.
\label{e6a}
\Ee
So, \eqref{e5} cannot hold unless $s\geq d/p-d$. Suppose now that $s=d/p-d$, and consider
$
g=\sum_{j=1}^N g_j\in\cF_1.
$
In this case we have 
$
\int_I g(x)\,dx\,=\,N
$
and
$
\|g\|_{F^s_{p,q}(\SR^d)}\lesssim N^{1/p}.
$
Thus, \eqref{e5} cannot hold unless $0<p\leq 1$. Observe that when  $d=1$ this completes the proof of the lemma.

Suppose now that $d\geq 2$. Define now the functions
\[
G_j(x_1,x'):=\,g_j(x_1)\,\chi(x'), \quad j\geq 1,
\] 
where $\chi\in C^\infty_c(0,1)^{d-1}$ has $\int\chi=1$. Then, $G_j\in\cF_2$ and $\int_IG_j=1$.
On the other hand
\[
\|G_j\|_{F^s_{p,q}(\SR^d)}\,\lesssim\,\|g_j\|_{F^s_{p,q}(\SR)}\,\lesssim\, 2^{-j(\frac1p-1-s)}.
\]
So \eqref{e5} can only hold if $s\geq 1/p-1$. In the case $s=1/p-1$, consider the function
$
G=\sum_{j=1}^NG_j\in\cF_2.
$
Then we have
\[
\int_IG=N\mand \|G\|_{F^s_{p,q}(\SR^d)}\,\lesssim\,\|g\|_{F^s_{p,q}(\SR)}\,\lesssim\, N^{1/p}.
\]
Thus, \eqref{e5} can only hold if $0<p\leq 1$. This does not give a new region if $d\geq 2$.
\end{proof}

\subsection{{\it  Extension of $\bbone_I$ as a bounded functional in $F^s_{p,q}$}} \label{addendum}

The next result is a converse of Lemma \ref{L_e1}, which in addition gives 
a continuous extension of the functional $\bbone_I$ to the whole space $F^s_{p,q}$.

\begin{lemma}
\label{L_e2}
Let $(s,p,q)$ be numbers satisfying  (a) or (b) in Lemma \ref{L_e1}, and let $I\in\sD_N$. Then, 
for each $f\in F^s_{p,q}$ the series
\Be
\bbone^*_I(f):=\sum_{j=0}^\infty \int_IL_j\La_j f
\label{1*I}
\Ee
is absolutely convergent, and there exists a constant $C_N=C_N(s,p,q)>0$ such that
\[
\big|\bbone^*_I(f)\big|\,\leq\, C_N\,\|f\|_{F^s_{p,q}},\quad \\ \,\text{ for all } f\in F^s_{p,q}.
\]
Moreover, 

\Benu
\item \eqref{1*I} does not depend on the specific decomposition $I=\sum_{j=0}^\infty L_j\La_j$, 
\item If $f\in F^s_{p,q}(\SR^d)$ is locally integrable with polynomial growth
 then
\[
\bbone^*_I(f)=\int_If(x)\,dx.
\]
\item For every $\zeta\in C^\infty_c$ such that $\zeta\equiv 1$ in a neighborhood of $\bar{I}$, it holds
\Be
\bbone^*_I(\zeta\, f)\,=\, \bbone_I^*(f),\,\, \text{for all }f\in F^s_{p,q}.
\label{e31}
\Ee
\Eenu
\end{lemma}
\begin{proof}
Since $F^s_{p,q}\hookrightarrow F^s_{p,\infty}$, it suffices to prove the result for $q=\infty$.
First notice that
\[
\Big|\sum_{j=0}^N\int_IL_j\La_j f\Big|\leq \int_I\big|\varPi_Nf\big|\leq|I|\,\big\|\varPi_Nf\|_{L^\infty(I)}.
\]
Now, using the Peetre maximal functions,
\[
\big\|\varPi_Nf\|_{L^\infty(I)}\lesssim \Big(\mint_{I^{**}}\big|\fM^*_N(\varPi_N f)\big|^p\Big)^{1/p} 
\lesssim 2^{\frac{Nd}p}\,\big\|\varPi_N f\big\|_p,
\]
and letting $r=\min\{p,1\}$, the last factor is bounded by\[
\big\|\varPi_N f\big\|_p\leq 
\Big(\sum_{j=0}^N\|L_j\La_jf\|_p^r\Big)^\frac1r\lesssim \Big(\sum_{j=0}^N2^{-jsr}\Big)^\frac1r\,\|f\|_{B^s_{p,\infty}}.
\]
So, we are left with proving that
\[
\sum_{j>N}\Big|\int_IL_j\La_j f\Big|\lesssim C_N\,\|f\|_{F^s_{p,\infty}}.
\]
Since $j>N$, we can use \cite[Lemma 2.3.ii]{gsu} and inequality \eqref{intwJ2} to obtain
\Beas
\Big|\int_IL_j\La_j f\Big| & \lesssim & 2^{-jd}\sum_{J\in\sD_{j+1}(\partial I)}\|\La_jf\|_{L^\infty(J)}\\
& \lesssim &  2^{-jd}\sum_{J\in\sD_{j+1}(\partial I)}\Big(\mint_{w(J)}\big|\fM^*_j(\La_j f)\big|^p\Big)^{1/p}.
\Eeas
In the case (b), {\it i.e.}  $s=d/p-d$ and $0<p\leq1$, we argue as in \eqref{intwJ} and obtain
\[
\sum_{j>N}\Big|\int_IL_j\La_j f\Big|^p\lesssim \int_{I^{**}}\sup_{\ell>N}\big|2^{\ell(\frac dp-d)}\fM^*_\ell(\La_\ell f)\big|^p\lesssim
\|f\|^p_{F^{ d/p-d}_{p,\infty}}.
\]
In the case (a), {\it i.e.} $s\,>\,\max\{\frac dp-d,\,\frac1p-1\}$, one can prove in a similar fashion the stronger estimate
\[
\sum_{j>N}\Big|\int_IL_j\La_j f\Big|^r\lesssim C_N(s,p)\,\|f\|^r_{B^s_{p,\infty}}
\]
with $r=\min\{p,1\}$. 

It is immediate to verify that \eqref{1*I} does not depend on the specific resolution of the identity $I=\sum_{j=0}^\infty L_j\La_j$.
The assertion (2) in the statement is a consequence of the convergence of 
the approximate identity $\varPi_n f\to f$ in $L^1_{\rm loc}(\SR^d)$ when $f$ is locally integrable with polynomial growth.

We finally verify the third assertion. Let $\varPi_n=\sum_{j=0}^n L_j\La_j$ and $\chi=1-\zeta$. Then,
\[
\big|\bbone^*_I(f)-\bbone^*_I(\zeta f)\big|=\lim_{n\to\infty}\Big|\int_I\varPi_n(\chi f)\Big|.
\]
Using distribution theory we can write
\[
\int_I\varPi_n(\chi f)(x)\,dx=\int_I\big\langle\chi f,\varPi_n(x-\cdot)\big\rangle\,dx =\Big\langle f,\chi\textstyle\int_I\varPi_n(x-\cdot)dx\Big\rangle.
\]
The result follows after checking that for 
\[
\Phi_n(y):=\chi(y)\,\int_I\varPi_n(x-y)dx
\] we have $\lim_{n\to \infty}\Phi_n= 0$ 
in the topology of the Schwartz class.
\end{proof}

Let $h\in\sH_d$. The previous result can be applied to define $h^*$ as a continuous linear functional in $F^s_{p,q}$. Namely,
\[
h^*(f)=\sum_{j=0}^\infty \int h(x) \,L_j\La_jf(x)\,dx,\quad \text{ for }f\in F^s_{p,q}.
\] 
Then, Lemmas \ref{L_e1} and \ref{L_e2} imply the following.

\begin{corollary}
The Haar functions, regarded as linear functionals, can be continuously extended from $\cS$  into  $F^s_{p,q}$ if and only if the indices $(s,p,q)$
satisfy (a) or (b) in Lemma \ref{L_e1}.
\end{corollary}

\

\section{Failure of density for $s=1$}
\label{S_dense}

\begin{proposition} 
\label{fdensity}
There exists a Schwartz function $f$ supported in $(\frac 1{16}, \frac{15}{16})^d$ such that,
for all $0<p\leq 1$ and $0<q\leq \infty$, 
\Be
\label{liminf}\liminf_{N\to \infty} \|\bbE_N f -f\|_{F^1_{p,q}}>0.\Ee
Moreover, if $d/(d+1)\leq p<1$ and $0<q\le 2$ then 
the span of $\sH_d$ is not a dense set in $F^1_{p,q}(\SRd)$. 
\end{proposition}
\begin{proof} 
The proof of \eqref{liminf} uses the same function $f$ as in \cite[Proposition 8.3]{gsu-endpt}. 
Namely, pick $\eta\in C^\infty_c(\SR^d)$ such that $\supp\eta\subset(\frac 1{16}, \frac{15}{16})^d$ and  $\eta(x)=1$ on $(1/8, 7/8)^d$.
Then consider $f(x)=x_1\,\eta(x)$. In \cite[Proposition 8.3]{gsu-endpt} it was  shown that this function satisfies
\Be
\liminf_{N\to \infty} \|\bbE_N f -f\|_{B^1_{p,\infty}}>0.
\label{liminfB}
\Ee
Therefore, \eqref{liminf} follows from here and the embeddings 
\[
F^1_{p,q}\hookrightarrow F^1_{p,\infty} \hookrightarrow B^1_{p,\infty}.
\]
We next show that \eqref{liminf} implies the failure of the density of span $\sH_d$ in $F^1_{p,q}$, for all $0<q\leq 2$ and $d/(d+1)\leq p<1$. 
Indeed, assume for contradiction that such density holds, and given $f\in\cS$ as in \eqref{liminf} and $\e>0$, find $g\in\Span\sH_d$ such that $\|f-g\|_{F^1_{p,q}}<\e$. 
Let $N_0$ be large enough so that $\SE_N(g)=g$ for all $N\geq N_0$. Then, the (quasi-)triangle inequality and the uniform boundedness of $\SE_N$ in Theorem \ref{th3} gives
\[
\|f-\SE_Nf\|_{F^1_{p,q}}\lesssim\|f-g\|_{F^1_{p,q}}+\|\SE_Ng-\SE_Nf\|_{F^1_{p,q}} \lesssim \e,\quad N\geq N_0,
\]  
which contradicts \eqref{liminf}.
\end{proof}

\begin{remark} It would be interesting to settle the question whether 
$\Span\sH_d$ is dense in the spaces $F^1_{p,q}$, when $d/(d+1)\leq p<1$ and $2<q< \infty$. As the operators $\bbE_N$ are not uniformly bounded in this range our current argument is not sufficient to give an answer (\cf. also \cite[\S8.1]{gsu-endpt} for a similar discussion about the Besov space analogue of this question).
\end{remark}

\section{Localization and partial sums of admissible enumerations}\label{localization}

Let $\cU=\{u_n\}_{n=1}^\infty$ be a strongly admissible enumeration of $\sH_d$, as in Definition \ref{strongly-adm} above. 
Explicit examples of such enumerations are not hard to construct; see e.g. \cite[\S 11]{gsu-endpt}.

Here we quote a localization lemma for such enumerations, which relates the partial operators $S_R^\cU$
and the dyadic averages $\SE_N$ and $T_N[\cdot, \fa]$. We let $\varsigma\in C^\infty_c$ be supported in a 
 $10^{-2}$ neighborhood of $[0,1)^d$ and so that 
 \Be\label{spacelocalization}
 \sum_{\nu\in \bbZ^d}\varsigma(\cdot-\nu)\equiv 1,\Ee
and denote $\varsigma_\nu=\varsigma(\cdot-\nu)$, $\nu\in \bbZ^d$. 
The following identity has been proved in \cite[Lemma 9.1]{gsu-endpt}.

\begin{lemma}  \label{basicdec} 
Let $\cU$ be a strongly admissible enumeration of $\sH_d $. Then, for every $R\in\SN$ and $\nu\in \bbZ^d$ there is an integer
$N_\nu=N_\nu(R)\ge -1$ and $\{0,1\}$-sequences $\fa^{\ka,\nu}$, $0\leq \kappa\leq b$,
such that for all $g\in L^1_{\rm loc}(\SR^d)$ we have
\Be
S_R^\cU [g\varsigma_\nu]  = \bbE_{N_\nu}[g\varsigma_\nu] +
\sum_{\kappa=0}^{b} T_{N_\nu+\ka}[
g\varsigma_\nu,\fa^{\ka,\nu}].
\label{SET}
\Ee
\end{lemma}

We next recall a localization property of the $F^s_{p,q}$-quasinorms; see \cite[2.4.7]{triebel2} (and \cite[2.4.2]{triebel4} for $p=\infty$). 

\begin{lemma}
\label{L_loc} 
Let $0<p, q\leq \infty$ and $s\in \bbR$. Then it holds
\Be\label{patching} \Big\|\sum_{\nu \in \bbZ^d} \vsig_\nu g\Big\|_{F^s_{p,q}} \,\approx\, \Big(\sum_{\nu \in \bbZ^d} 
\big\| \vsig_\nu\,g\big\|_{F^s_{p,q}}^p\Big)^{1/p}.
\Ee
\end{lemma}

We are now ready to prove the uniform boundedness of the operators $S^\cU_R$. 
We assume that $(s,p,q)\in \fA$, as defined in Definition \ref{Adef},
so that these operators 
can be continuously extended to the whole space $F^s_{p,q}$. More precisely, if $p,q<\infty$, condition \eqref{BBdual} 
holds and $S_R$ is well-defined as in section 1 (that is, extended from $\cS$ to $F^s_{p,q}$ by density).
In order to include as well the cases $p=\infty$ or $q=\infty$, one first considers extensions
of the dual functionals $u^*_n$ to the full space $F^s_{p,q}$ as follows
\Be
u^*_n(f):=\sum_{j=1}^\infty 2^{k(n)d}\int h^{\epsilon(n)}_{k(n),\nu(n)}\, L_j\La_j f, \text{ for } f\in F^s_{p,q};
\label{un*}
\Ee
see the details in \S\ref{addendum}. In this way, the identity in \eqref{SET} remains valid for all $g\in F^s_{p,q}$.

\begin{proposition} \label{8.2}
Let $(s,p,q)\in\fA$. Suppose 
that 
\Be
\label{supEN}
\sup_{N\geq 0} \|\bbE_N\|_{F^s_{p,q}\to F^s_{p,q}} 
+
\sup_{N\geq 0}\sup_{\|\fa\|_{\ell^\infty}\le 1} \|T_N[\cdot, \fa]\|_{F^s_{p,q}\to F^s_{p,q}}  <\infty.
\Ee
Then, for every strongly admissible enumeration $\cU$ it holds
\[
\sup_{R\geq 1}\,\big\|S^{\,\cU}_R\|_{F^s_{p,q}\to F^s_{p,q}} <\infty.
\]
\end{proposition}

\begin{proof}
Consider $S_R=S_R^\cU$ as a continuous operator in $F^s_{p,q}$ (as described in \S\ref{addendum}).
Then, the support properties of the extension, see \eqref{e31}, imply that
\[
\varsigma_{\nu'} S_R (f\varsigma_\nu)=0,\quad \mbox{whenever}\quad |\nu-\nu'|_\infty\geq3.
\]
Then, using \eqref{spacelocalization} and \eqref{patching}, 
 \begin{align*}
 \big\|S_R f\big\|_{F^s_{p,q}} & \approx  \Big(\sum_{\nu'}\big\|\varsigma_{\nu'} S_R\big(\sum_{\nu}\varsigma_\nu f\big)\big\|_{F^s_{p,q}}^p\Big)^\frac1p
\\&\lesssim \Big( \sum_{\nu'} \sum_{\nu\,:\,|\nu-\nu'|_\infty\leq2}\Big\| 
  \varsigma_{\nu'}S_R(f\varsigma_\nu)\Big\|_{F^s_{p,q}}^p \Big)^{1/p}
 \\
 &\lesssim \Big( \sum_{\nu} \big\| S_R(f\varsigma_\nu)\big\|_{F^s_{p,q}}^p \Big)^{1/p},
 \end{align*}
using in the last step that $\varsigma_{\nu'}$ is a uniform multiplier in $F^s_{p,q}$; see \cite[4.2.2]{triebel2}.
Then Lemma  \ref{basicdec}  and \eqref{supEN} 
give
 \begin{align*}
 \big\|S_R f\big\|_{F^s_{p,q}} &\lesssim  \Big( \sum_{\nu} \big\| \bbE_{N_\nu}(f\varsigma_\nu)\big\|_{F^s_{p,q}}^p 
+
\big\| \sum_{\kappa=0}^b T_{N_\nu+\kappa}[f\varsigma_\nu, \fa^{\ka,\nu}]\big\|_{F^s_{p,q}}^p \Big)^{1/p} 
 \\
 &\lc_b 
 \Big( \sum_{\nu} \big\|f\varsigma_\nu\big\|_{F^s_{p,q}}^p \Big)^{1/p} \, \approx\, \|f\|_{F^s_{p,q}}. \qedhere
 \end{align*}
\end{proof} 

\begin{remark}
The equivalence in \eqref{patching} is also true with $\vsig$ replaced by $\bbone_{[0,1]^d}$
when
\[
\max\Big\{\frac dp-1, \frac1p-1\Big\}<s<\frac1p,
\]
as in that case characteristic functions of cubes are multipliers in $F^s_{p,q}$.
In particular, for those indices the assertion in Proposition \ref{8.2} holds as well with the weaker notion of \emph{admissible} enumeration; see
\cite[\S3]{gsu}. This is in particular the case when $s=1$ and $d/(d+1)<p<1$.
\end{remark}
 
Finally, we conclude with the following observation, which we shall use to transfer negative results between the operators $\SE_N$ and $S_R$.
The explicit construction is given in \cite[\S11]{gsu-endpt}.

\begin{lemma}\label{L_SRm}
There exists a strongly admissible enumeration $\cU$ with the following property: for every $m\geq 0$ there exists an integer $R(m)\geq 1$ such that
\Be
S^\cU_{R(m)}f\,=\,\SE_m f, \quad f\in C^\infty_c((-5,5)^d).
\label{SREm}
\Ee
\end{lemma}

\section{The Schauder basis property: proof of Theorem \ref{th1}}
\label{S_th1}

\subsection{\it Necessary conditions}

Suppose that every strongly admissible enumeration $\cU$ of $\sH_d$ is a Schauder basis of $F^s_{p,q}$.
This implies that $\Span\sH_d$ must be dense (hence $p,q<\infty$), and 
\Be
C_{\cU}:=\sup_{R\geq1}\big\|S^\cU_R\big\|_{F^s_{p,q}\to F^s_{p,q}}<\infty.
\label{CcU}
\Ee
Moreover, if we select $\cU$ as in Lemma \ref{L_SRm},
then we must have
\[
\sup_{m\geq0}{\rm Op}_\cS\big(\SE_m, F^s_{p,q}\big)\,\leq \,C_{\cU}<\infty.
\]  
In view of Proposition \ref{fdensity} and Corollary \ref{C_e} this is only possible if (i), (ii) or (iii) in Theorem \ref{th1} hold.

\subsection{\it Sufficient conditions}

Under the assumptions in (i), (ii), and (iii) of Theorem \ref{th1}, the operators $\SE_N$ and $T_N[\cdot, \fa]$ are uniformly bounded in $F^s_{p,q}$, by Theorem \eqref{th3}. So we can use Proposition \ref{8.2} and conclude that \eqref{CcU} must hold.
The density of $\Span\sH_d$ is also true in this range, so we conclude that $\cU$ is a Schauder basis of $F^s_{p,q}$.

\subsection{\it Consequences for the basic sequence property}
Theorem \ref{th3} additionally implies convergence of  basic sequences in the cases when $\Span\sH_d$ is not dense.
Namely, when $p=\infty$ or $q=\infty$, let $f^s_{p,q}$ denote the closure of the $\cS$ in $F^s_{p,q}$.
When $s<1/p$ the subset $\Span \sH_d$ is dense in  $f^s_{p,q}$, so we deduce the following.

\begin{corollary}
Let $(s,p,q)$ be as in (i), (iii) or (iv) in Theorem \ref{th3}.
Then, every admissible enumeration $\cU$ is a Schauder basis of $f^s_{p,q}$.
That is,
\[
f=\sum_{n=1}^\infty u^*_n(f) u_n, \quad \text{for all }\;f\in f^s_{p,q},
\]
with convergence in the norm of $F^s_{p,q}$.
\end{corollary}
\begin{remark}
Observe that we have excluded the cases (ii) and (v) in Theorem \ref{th3}.
In these cases we can only say that $\cU$ is a Schauder basis of the subspace
\[
{\overline{\Span\sH_d}}^{F^s_{p,q}}.
\]
A precise description of this subspace in those cases, however, is not clear.
In the range (ii), ie $s=1$ (and $q\leq 2$) this subspace cannot contain the Schwartz class $\cS$, as shown by Proposition \ref{fdensity}. 
On the other hand, in the case (v), {\it i.e.} $s=0$ and $p=\infty$,
 this subspace \emph{strictly} contains $f^0_{\infty,q}$. 
Indeed, first of all one has
\[
f^0_{\infty,\infty}\cap \Span\sH_d\;=\;\{0\}; 
\]
see \cite[Proposition 5.1]{gsu-endpt}. 
Next, for all $q\le \infty$, the inclusion \[C^\infty_c(\SR^d)\subset {\overline{\Span\sH_d}}^{F^0_{\infty,q}} \]
follows, when $d=1$, from the elementary embedding $B^{1/p}_{p,\infty}(\SR)\hookrightarrow F^0_{\infty,q}(\SR)$ and the 
corresponding result for $B^{1/p}_{p,\infty}(\SR)$ in \cite[Proposition 8.6]{gsu-endpt}.
When $d\geq 2$, one can approximate each $f\in C^\infty_c(\SR^d)$ by a linear combination of functions $g^1(x_1)\cdots g^d(x_d)$
with  $g^i\in C^\infty_c(\SR)$, and then use the previous result.
\end{remark}

\section{The unconditional basis property: proof of Theorem \ref{th_unc}}
\label{S_th_unc}

The fact that $\sH_d$ is an unconditional basis of $F^s_{p,q}$ when \eqref{range1} and \eqref{q_range} hold was shown by Triebel in  \cite[Theorem 2.21]{triebel-bases}.
We now indicate references for the \emph{negative} end-point results, corresponding to the dotted or dashed lines around the green region in Figure \ref{fig_unc}.

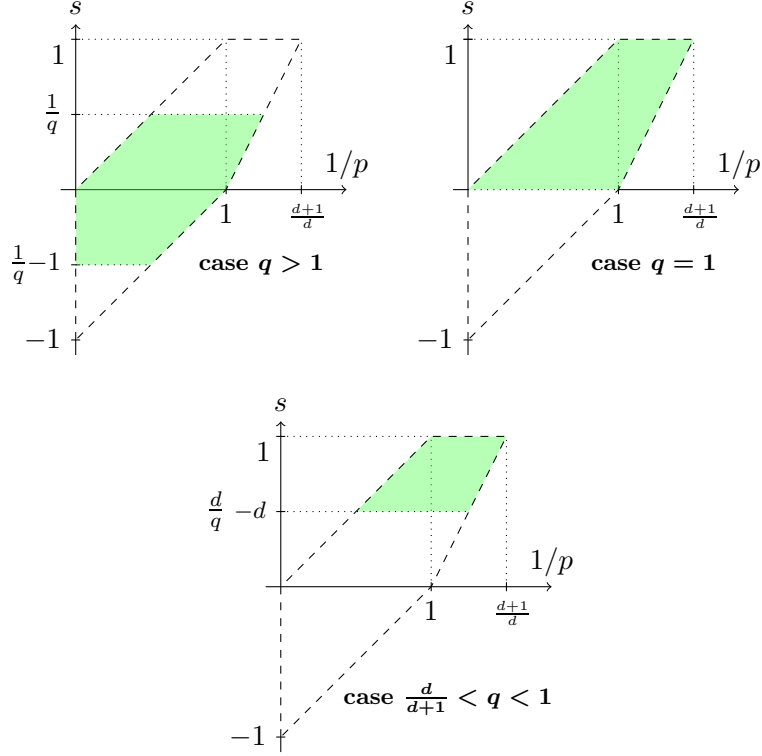
\begin{figure}[h]
 \centering
\subfigure
{\begin{tikzpicture}[scale=2]

\node [right] at (0.75,-0.5) {{\footnotesize {\bf case $\boldsymbol{q>1}$}}};

\draw[->] (-0.1,0.0) -- (1.8,0.0) node[above] {${1}/{p}$};
\draw[->] (0.0,-0.0) -- (0.0,1.1) node[above] {$s$};
\draw (0.0,-1.1) -- (0.0,-1.0)  ;

\draw (1.0,0.03) -- (1.0,-0.03) node [below] {$1$};
\draw (1.5,0.03) -- (1.5,-0.03) node [below] {{\tiny $\;\;\frac{d+1}{d}$}};
\draw (0.03,1.0) -- (-0.03,1.00);
\node [left] at (0,0.9) {$1$};
\draw (0.03,.5) -- (-0.03,.5) node [left] {$\tfrac{1}{q}$};
\draw (0.03,-.5) -- (-0.03,-.5) node [left] {$\tfrac{1}{q}${\small{$-1$}}};
\draw (0.03,-1.0) -- (-0.03,-1.00) node [left] {$-1$};

\draw[dotted] (1.0,0.0) -- (1.0,1.0);
\draw[dotted] (0,1.0) -- (1.0,1.0);
\draw[dotted] (1.5,0.0) -- (1.5,1.0);

\path[fill=green!70, opacity=0.4] (0.0,0.0) -- (.5,.5)-- (1.25,0.5) -- (1,0)--(.5,-.5) -- (0,-0.5)--(0,0);
\draw[dotted] (0,0.5)--(1.25,0.5);
\draw[dotted] (0,-0.5)--(0.5,-0.5);

\draw[dashed] (0.0,-1.0) -- (0.0,0.0) -- (1.0,1.0) -- (1.5,1.0) -- (1.0,0.0) --
(0.0,-1.0);

\end{tikzpicture}
}
\subfigure
{
\begin{tikzpicture}[scale=2]

\node [right] at (0.75,-0.5) {{\footnotesize {\bf case $\boldsymbol{q=1}$}}};

\draw[->] (1,0.0) -- (1.8,0.0) node[above] {${1}/{p}$};
\draw[->] (0.0,-0.0) -- (0.0,1.1) node[above] {$s$};
\draw (0.0,-1.1) -- (0.0,-1.0)  ;
\draw (-0.1,0) -- (0.0,0)  ;

\draw (0.03,1.0) -- (-0.03,1.00);
\node [left] at (0,0.9) {$1$};
\draw (1.0,0.03) -- (1.0,-0.03) node [below] {$1$};
\draw (1.5,0.03) -- (1.5,-0.03) node [below] {{\tiny $\;\;\frac{d+1}{d}$}};
\draw (0.03,1.0) -- (-0.03,1.00);
\draw (0.03,-1.0) -- (-0.03,-1.00) node [left] {$-1$};

\draw[dotted] (1.0,0.0) -- (1.0,1.0);
\draw[dotted] (0,1.0) -- (1.0,1.0);
\draw[dotted] (1.5,0.0) -- (1.5,1.0);

\path[fill=green!70, opacity=0.4] (0.0,0) -- (1,1)-- (1.5,1) -- (1,0)--(0,0);
\draw[dotted] (0,0)--(1,0);

\draw[dashed] (0.0,-1.0) -- (0.0,0.0) -- (1.0,1.0) -- (1.5,1.0) -- (1.0,0.0) --
(0.0,-1.0);

\end{tikzpicture}
}

\subfigure
{
\begin{tikzpicture}[scale=2]

\node [right] at (0.35,-0.75) {{\footnotesize {\bf case $\boldsymbol{\frac d{d+1}<q<1}$}}};

\draw[->] (-0.1,0.0) -- (1.8,0.0) node[above] {${1}/{p}$};
\draw[->] (0.0,-0.0) -- (0.0,1.1) node[above] {$s$};
\draw (0.0,-1.1) -- (0.0,-1.0)  ;

\draw (1.0,0.03) -- (1.0,-0.03) node [below] {$1$};
\draw (1.5,0.03) -- (1.5,-0.03) node [below] {{\tiny $\;\;\frac{d+1}{d}$}};
\draw (0.03,1.0) -- (-0.03,1.00);
\node [left] at (0,0.9) {$1$};
\draw (0.03,.5) -- (-0.03,.5) node [left] {$\tfrac{d}{q}$ {\footnotesize $-d$}};
\draw (0.03,-1.0) -- (-0.03,-1.00) node [left] {$-1$};

\draw[dotted] (1.0,0.0) -- (1.0,1.0);
\draw[dotted] (0,1.0) -- (1.0,1.0);
\draw[dotted] (1.5,0.0) -- (1.5,1.0);

\path[fill=green!70, opacity=0.4] (.5,.5)-- (1,1)--(1.5,1) -- (1.25,0.5)--(0.5,0.5);
\draw[dotted] (0,0.5)--(1.25,0.5);

\draw[dashed] (0.0,-1.0) -- (0.0,0.0) -- (1.0,1.0) -- (1.5,1.0) -- (1.0,0.0) --
(0.0,-1.0);

\end{tikzpicture}
}
\caption{Parameter domain for unconditionality in the cases $q>1$, $\;q=1$ and $d/(d+1)<q<1$, respectively.}\label{fig_unc}
\end{figure}

The trivial cases correspond to the lines $p=\infty$, $s=1/p$, and to the line $s=1/p-1$ with $p>1$. 
In all of them not even the Schauder basis property may hold. Namely, if $p=\infty$ then $F^s_{\infty, q}$ is not separable, and hence 
$\Span\sH_d$ is not dense (see however Remark \ref{R_unc} below for the validity of unconditionality in the subspace $f^s_{\infty,q}$). 
The other two cases are excluded because $(s,p,q)\not\in\fA$, and hence \eqref{BBdual} fails.


Concerning the horizontal line $s=1$, this is a borderline of the unconditionality region when $d/(d+1)\leq p, q\leq 1$.  
This case is excluded by Proposition \ref{fdensity}, since $\Span\sH_d$ is not dense in $F^1_{p, q}$, so 
the Schauder basis property
 cannot hold here. 

At the line $s=d-d/p$, for $d/(d+1)< p \leq 1$, we have a positive Schauder basis result for strongly admissible $\cU$,  by Theorem \ref{th1}.
So we must prove that such a basis cannot be unconditional in $F^{s}_{p,q}$. This was already  shown in \cite[Theorem 13.1]{gsu-endpt}, based on an explicit example which works  well 
in both the Besov and the Triebel-Lizorkin setting.  

Finally, we consider the horizontal lines of the green region which lie inside the open pentagon $\fP$.
In \cite{su}, the failure of unconditionality in these lines was shown in the case $q>1$ and $d=1$, indeed  for all exponents  $p\geq d/(d+1)$ (by \cite[Remark 7.1]{su}). Here 
we show how to modify the arguments in that paper to cover as well the cases $q\leq1$, and extend the construction to all $d\geq1$. 

We recall some notation from \cite{su}. To each finite set $E\subset\sH_d$ we associate the projection operator
\[
P_E(f)=\sum_{h\in E}\langle f, h^*\rangle h,
\]
where $h^*=2^{kd} h$ is the dual functional of a Haar function $h\in\sH_d$ of frequency $2^k$. We also write HF$(E)$ for the set of all Haar frequencies $2^k$ of elements $h\in E$. 

We first remark that the results in \cite[\S6]{su} remain valid when $q\leq 1$.
Namely, for each $N\geq2$, an explicit construction is given of a function $f=f_N\in F^{1/q-1}_{p,q}(\SR)$ and a set $E=E_N\subset\sH_1$ with $\#{\rm HF}(E)\leq N4^N$ such that\footnote{In the notation of \cite[\S6]{su}, one should consider sets $A$ of \emph{consecutive} Haar frequencies, so that the associated ``density'' number in \cite[(43)]{su} takes the value $Z=N$.}
\Be
\|f_N\|_{F^{1/q-1}_{p,q}(\SR)}\lesssim N^{1/q} \mand \|P_{E_N}(f_N)\|_{F^{1/q-1}_{p,q}(\SR)} \geq N^{1+\frac1q},
\label{fNR}
\Ee
if $0<q\leq p<\infty$. In particular, for $d=1$,
\[
\|P_{E_N}\|_{F^{1/q-1}_{p,q}\to F^{1/q-1}_{p,q}}\,\gtrsim \, N,
\]
and hence $\sH_1$ is not unconditional at the lower segment of the green region in Figure \ref{fig_unc}.

%

When $d\geq 2$, the above example can be adapted in two different ways. 
If $1<q<\infty$, one considers the tensorized functions 
\[
F_N(x_1,x'):= f_N(x_1)\otimes \chi(x'),
\] 
where $f_N$ is as in \eqref{fNR} and  $\chi\in C^\infty_c((-1,2)^{d-1})$ with $\chi\equiv1$ in $[0,1]^{d-1}$, and defines the sets
\[
\cE_N:=\Big\{ h\otimes \bbone_{[0,1]^{d-1}}\mid h\in E_N\Big\}.
\] 
Then, a standard computation (as in \S\ref{S_ExFd} above) 
 gives
\[
 \|F_N\|_{F^{1/q-1}_{p,q}(\SR^d)}\lesssim N^{1/q} \mand \|P_{\cE_N}(F_N)\|_{F^{1/q-1}_{p,q}(\SR^d)} \gtrsim N^{1+1/q}.
\]

When $q\leq 1$, one considers instead the natural generalization to $\SR^d$ of the construction in \cite[\S6]{su}, namely using the test function
$\prod_{i=1}^d\eta(x_i)$,
in place of the one dimensional function $\eta$ in \cite[(46)]{su}.
More precisely, if 
 $ 1\leq \kappa \leq 4^{Nd}$,
 $b_\kappa=\kappa N$ and $1\leq \sigma\leq N$, one defines the functions
\[
\cY_{\kappa,\sigma}(x_1,\ldots, x_d)=
\sum_{{(\nu_1,\ldots,\nu_d)\in\SZ^d}\atop{0\leq \nu_i<2^{b_\kappa-N-2}}}
2^{-\sigma d}\,
\prod_{i=1}^d\eta\big(2^{b_\kappa+N-\sigma}(x_i-2^{N+2-b_\kappa}\nu_i)\big)
\]
and, for $t\in[0,1]$,
\[
f_{t}=\sum_{\kappa=1}^{4^{Nd}}r_\kappa(t)\,2^{-(b_\kappa+N)(\frac dq-d)}\,\sum_{\sigma=1}^N \,\cY_{\kappa, \sigma},
\]
where $r_k(t)$ is a Rademacher function.
Then, arguing as in \cite[Lemma 6.2]{su} and \cite[Proposition 6.3]{su}, if $0<q\leq p$ one verifies that
\[
\|f_{t}\|_{F^{d/q-d}_{p,q}(\SR^d)}\lesssim \,N^{1/q},
\]
and that for some $t_0$ and some $E\subset \sH_d$ with HF$(E)\subset \{2^k\}_{ 1\leq k\leq N4^{dN}}$,
\[  
\big\|P_E(f_{t_0})\big\|_{F^{d/q-d}_{p,q}(\SR^d)}\gtrsim\,N^{1+\frac1q}.
\]
This completes the proof of Theorem \ref{th_unc}.
\ProofEnd

\begin{remark}
\label{R_unc}
If $p=\infty$ one can  ask whether the Schauder basis property in the subspace $f^s_{\infty,q}$, for $-1<s< 0$, can be upgraded to unconditional basis.
This is certainly true when $1/q-1<s<0$, by the uniform boundedness of 
the projection operators $P_E$ (which follows by duality from the corresponding result for $F^{-s}_{1,q'}$), and by the density of $\Span\sH_d$ in $f^s_{\infty,q}$.

We now show that at the endpoint $s=1/q-1$ unconditionality must fail. If not, the operators $P_E$ would be uniformly bounded
in $f^{-1/q'}_{\infty, q}$, for all finite $E\subset\sH_d$. Fix $p_0\in(q,\infty)$, and for each $N\geq 1$, pick a set $A=A_N\subset\{2^n\}_{n\geq0}$ 
of cardinality $2^N$ and such that $\log_2A$ is $N$-separated. 
Then, by \cite[Theorem 1.4]{su}, for every $E\subset \sH_d$ with $\#{\rm HF}\,(E)\subset A$ it holds
\Be
\|P_{E}\|_{F^{-1/q'}_{p_0,q}}\,\lesssim\, N^{1/q'}.
\label{Pfq}
\Ee
By interpolation one then has, for $\theta\in(0,1)$,
\[
\|P_{E}\|_{F^{-1/q'}_{p_0/\theta,q}}\,\lesssim\, \|P_{E}\|^{\theta}_{F^{-1/q'}_{p_0,q}}\,\lesssim\, N^{\frac{\theta}{q'}}.\quad 
\] 
But this contradicts the lower bound $N^{1/q'}$ (for the supremum of all such sets $E$) asserted in  \cite[Theorem 1.4.ii]{su}. Similar arguments also disprove the unconditionality for $s$ below the critical $1/q-1$.

Finally consider the space $f^0_{\infty,q}$ for $1\le q<\infty$, on which unconditionality fails (since otherwise it would hold on its dual $F^0_{1,q'}$,
see \cite[\S2.1.5]{RuSi96},  on which unconditionality fails by \cite[Prop. 13.3]{gsu-endpt}).
\end{remark}

\begin{remark} We now consider the 
spaces $f^s_{p,\infty}$, when $1<p<\infty$.
The Schauder basis property holds for $1/p-1< s<1/p$ while the unconditional basis property holds only for $1/p-1<s<0$, already by the estimates in \cite{triebel-bases}.  The unconditional basis property does not hold on $f^0_{p,\infty}$ since by duality (\cite[\S2.1.5]{RuSi96}) it would imply it on $F^0_{p',1}$ where it fails by \cite{su}. Finally when $p=q=\infty$ then $F^s_{\infty,\infty}=B^s_{\infty,\infty}$
hence $f^s_{\infty,\infty}=b^s_{\infty,\infty}$, and the unconditional basis property holds for $-1<s<0$ (for the dual statement see \cite{gsu-endpt}).
\end{remark}

\begin{remark} It would be interesting to investigate the   question of unconditionality of the Haar system as a basic sequence in $B^1_{p,q}$ and $F^1_{p,q}$ when $d/(d+1)< p<1$.
\end{remark}

\end{document}